\documentclass[12pt,oneside,reqno]{amsart}
\usepackage{amsmath,amssymb, amsthm,mathtools}
\usepackage[margin=1.5in]{geometry}
\usepackage[parfill]{parskip}
\usepackage{subfig}
\usepackage{caption}
\usepackage{subcaption}
\usepackage{nicefrac}
\usepackage{xcolor}

\newcommand{\paren}[1]{\left( #1 \right)}
\newcommand{\brac}[1]{\left[ #1 \right]}
\newcommand{\abs}[1]{\left\vert#1\right\vert}

\newcommand{\set}[1]{\left\{#1\right\}}
\newcommand{\R}{\mathbb{R}}
\newcommand{\new}{}%[1]{{\color{red}#1}}
\newcommand{\proj}{\mathrm{proj}}

\usepackage{amsfonts}
\usepackage{amsmath}
\usepackage{amssymb}
\usepackage{amsthm}
\usepackage{mathtools}
\usepackage{enumitem}
\newtheorem{theorem}{Theorem}
\newtheorem{lemma}[theorem]{Lemma}
\newtheorem{definition}[theorem]{Definition}
\newtheorem{proposition}[theorem]{Proposition}

\newtheorem{corollary}[theorem]{Corollary}

\newcommand{\T}{\mathbb T}
\newcommand{\Z}{\mathbb Z}

\newcommand{\dist}{\mathrm{dist}}

\begin{document}

\title{Voronoi Percolation: Topological Stability and Giant Cycles}
\author{Benjamin Schweinhart}
\email{bschwei@gmu.edu}
\address{Department of Mathematical Sciences, George Mason University, Fairfax, VA 22030, USA}
\author{Morgan Shuman}
\email{mshuman4@gmu.edu}
\address{Department of Mathematical Sciences, George Mason University, Fairfax, VA 22030, USA}

\begin{abstract}
We study the topological stability of Voronoi percolation in higher dimensions. We show that slightly increasing $p$ allows a discretization that preserves increasing topological properties with high probability. This strengthens a theorem of Bollobás and Riordan and generalizes it to higher dimensions. As a consequence, we prove a sharp phase transition for the emergence of $i$-dimensional giant cycles in Voronoi percolation on the $2i$-dimensional torus.
\end{abstract}

\maketitle

\begin{figure}[h]
    \centering
    \includegraphics[width=0.4\linewidth]{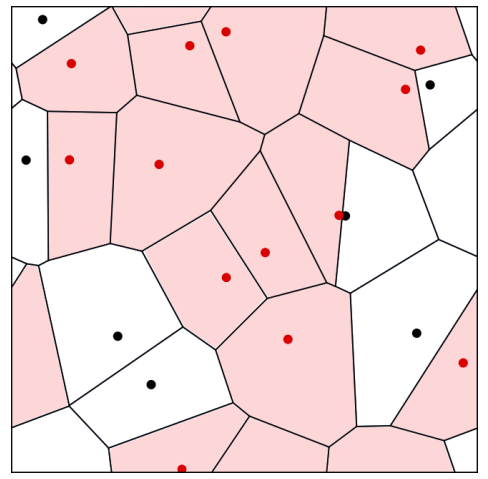}
    \caption{\label{figure:VoronoiPercolation} Voronoi percolation at $p = 0.5,$ shown with periodic boundary conditions.}
    \label{fig:enter-label}
\end{figure}

\section{Introduction}
\noindent The Poisson--Voronoi mosaic on $\mathbb{R}^d$ is the random polyhedral complex obtained by taking the Voronoi diagram of a Poisson process on $\mathbb{R}^d$ with constant intensity. Voronoi percolation is the random subcomplex of the mosaic where each Voronoi cell is included independently with probability $p.$ A seminal paper by Bollobás and Riordan shows that the percolation threshold of Voronoi percolation on $\mathbb{R}^2$ is $1/2$ \cite{Bollob_s_2005}. We prove a higher-dimensional, finite-volume analogue of this statement for Voronoi percolation on a torus: the $i$-dimensional homological percolation threshold on the $2i$-dimensional torus is $1/2.$ Roughly speaking, homological percolation occurs if there is an $i$-dimensional giant cycle spanning the torus. When $i=1,$ a giant cycle is a periodic path. Homological percolation was introduced by Bobrowski and Skraba~\cite{bobrowski2020homological,bobrowski2022homological} and further studied by Duncan, Kahle, and Schweinhart~\cite{duncan2025homological}.

Our proof strategy combines the ideas of~\cite{Bollob_s_2005} and ~\cite{duncan2025homological}: we construct a discretization of Voronoi percolation at a carefully chosen scale, prove a sharp phase transition for a ``stable'' homological percolation event in the discrete model, and show it coincides with the transition for the original model. A key technical lemma of~\cite{Bollob_s_2005} posits (roughly) that any paths lost in the discretization process can be made up by slightly increasing $p$. We strengthen this result by considering a broader class of topological conditions and generalizing to arbitrary dimensions. As a bonus, we obtain a shorter, more readable proof by using results from the literature on Delaunay triangulations \cite{boissonnat2013stability, bonnet2018maximaldegreepoissondelaunaygraph}. 

% Homological percolation was defined and investigated by Bobrowski and Skraba in \cite{bobrowski2020homological,bobrowski2022homological} and further studied by Duncan, Kahle, and Schweinhart in \cite{duncan2025homological}.  In a nutshell, homological percolation is the event that there is an $i$-dimensional random giant ``surface'' that is non-trivial in the homology of the torus. This is a direct generalization of the existence of a periodic path in the classical setting.

\begin{figure}[ht]
    \centering
    \includegraphics[width=0.5\linewidth]{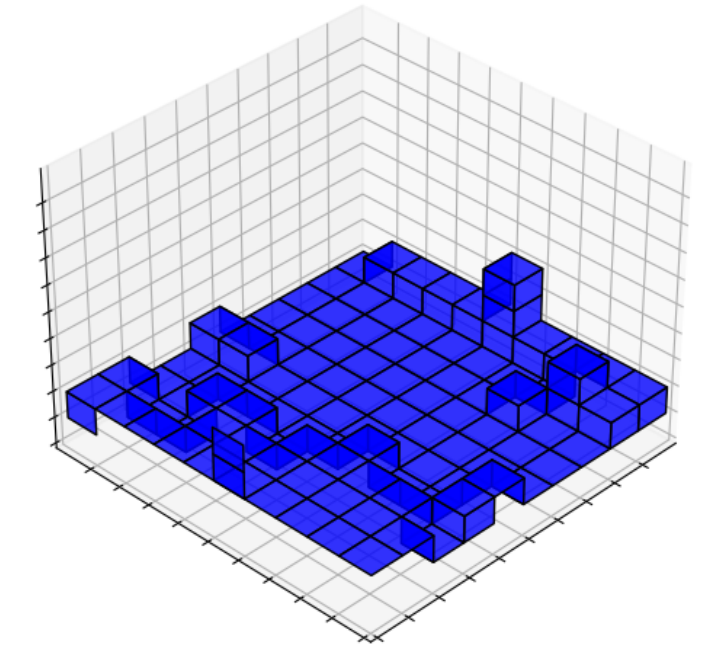}
    \caption{A giant cycle in $2$-dimensional plaquette percolation. Reproduced from \cite{duncan2025homological} with permission.}
    \label{fig:giant2cycle}
\end{figure}

% We examine two homological events: the existence of an $i$-dimensional spanning surface, called a \textbf{giant $i$-cycle}; and the presence of a basis of (non-homologous) giant $i$-cycles. We denote these events $A$ and $S$, respectively. To illustrate the significance of $S$, consider $A$ and $S$ for giant $1$-cycles on a $2$-dimensional torus. If there is only a single periodic path, say from bottom-to-top, then $A$ occurs but $S$ does not. If there are two non-homologous periodic paths, say one from left-to-right and one from top-to-bottom, then both $A$ and $S$ occur.

We examine two homological events: the existence of an $i$-dimensional spanning surface, called a \textbf{giant $i$-cycle}; and the presence of a basis of (non-homologous) giant $i$-cycles. We denote these events $A$ and $S$, respectively. Figure~\ref{fig:AS} shows a percolation where $A$ occurs but not $S.$ To illustrate the significance of $S$, consider $A$ and $S$ for giant $1$-cycles on a $2$-dimensional torus. If there is only a single periodic path, say from bottom-to-top, then $A$ occurs but $S$ does not. If there are two non-homologous periodic paths, say one from left-to-right and one from top-to-bottom, then both $A$ and $S$ occur. 
For discrete percolation processes, it isn't difficult to show that the events $A$ and $S$ have sharp threshold functions. However, it is challenging to prove that they coincide for the two events and that they do not depend on the size of the torus. Duncan, Kahle, and Schweinhart~\cite{duncan2025homological} investigated homological percolation in models of plaquette percolation and permutohedral site percolation. They demonstrated the existence of sharp phase transitions for $A$ and $S$ for $i$-dimensional homological percolation on a $2i$-dimensional torus (for which the threshold is $1/2$) and $(d-1)$-dimensional homological percolation on the $d$-dimensional torus (where the threshold is dual to that of classical percolation).  We show the former statement for $i$-dimensional Voronoi percolation on the $2i$-dimensional torus. Our proof strategy combines the ideas of~\cite{Bollob_s_2005} and ~\cite{duncan2025homological}: we construct a discretization of Voronoi percolation at a carefully chosen scale, prove a sharp phase transition for a ``stable'' homological percolation event in the discrete model, and show it coincides with the transition for the original model. A key technical lemma of~\cite{Bollob_s_2005} posits (roughly) that any paths lost in the discretization process can be made up  by slightly increasing $p$. We strengthen this result by considering a broader class of topological conditions and generalizing to arbitrary dimensions. As a bonus, we obtain a shorter, more readable proof by using results from the literature on Delaunay triangulations \cite{boissonnat2013stability, bonnet2018maximaldegreepoissondelaunaygraph}. 

\begin{figure}[t]
    \centering
    \includegraphics[width=0.5\linewidth]{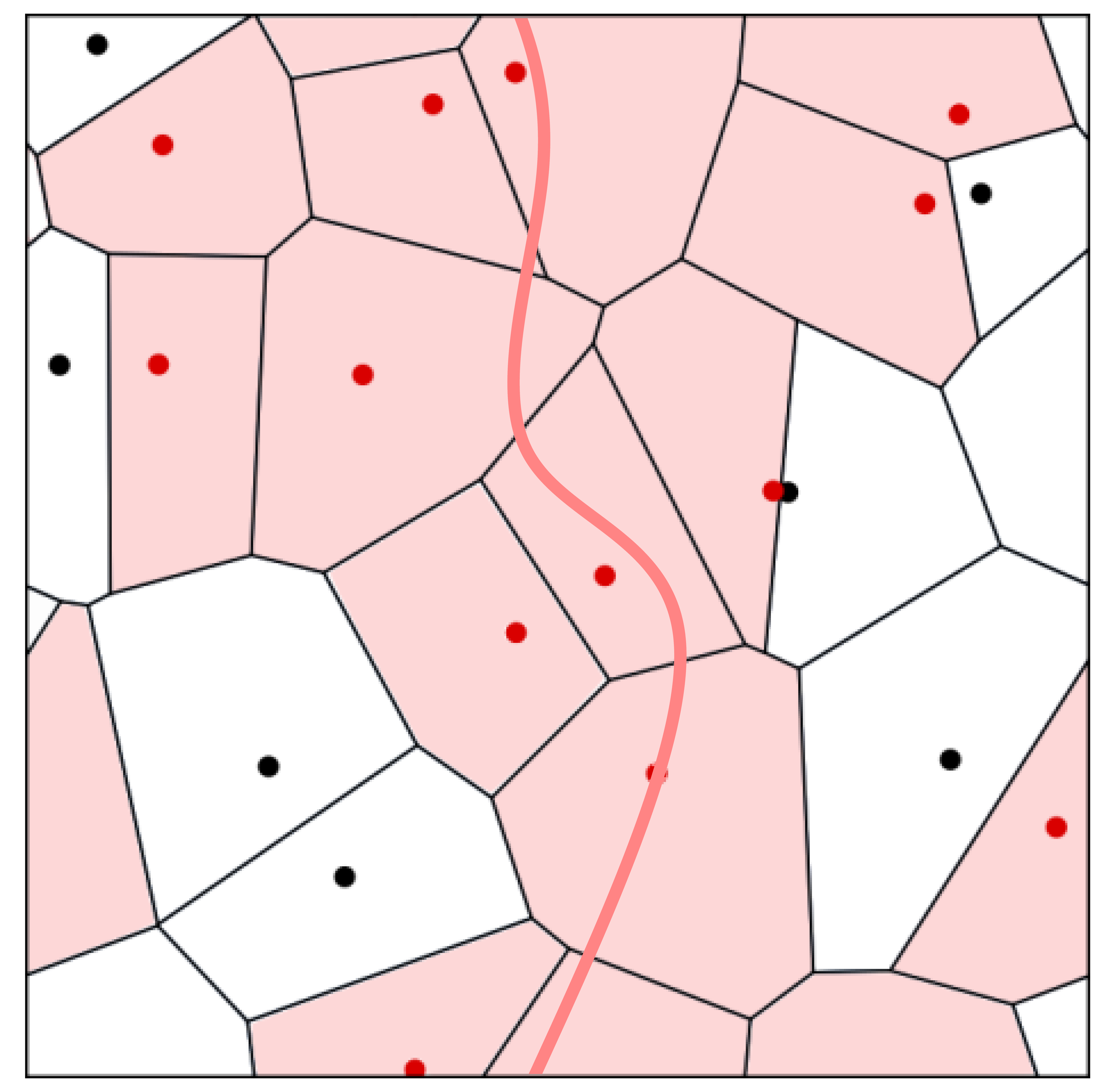}
    \caption{A percolation where $A$ occurs, but not $S$. Both the red and white Voronoi cells contains a giant $1$-cycle. The giant $1$-cycle in the red cells is highlighted.}
    \label{fig:AS}
\end{figure}

We quickly introduce the necessary notation to state our results and revisit them later. Let $\mathbb{T}^d_N$ be the cube $[-N,N]^d$ with opposite faces identified. Define $Z$ to be a Poisson point process of intensity $1$ on $\T^d_N$; color the points of $Z$ red independently with probability $p$, and color the remaining points white. Denote by $R=R(p)$ the set of red points of $Z$ and denote the Voronoi percolation as the pair $(Z,R)$. For $W\subset Z$, let $U(W,Z)$ be the union of of the Voronoi cells generated by the points of $W$ in $Z$. 

Roughly speaking, a point set $W \subset Z$ is \textbf{$\delta$-good} if the topology of the embedding of $U\paren{W,Z}$ into the torus is invariant to $\delta$-perturbations, where a $\delta$-perturbation is a bijection so that $\|z - \zeta(z)\| < \delta$ for all $z \in Z$. More precisely, we require that every $\delta$-perturbation $\zeta : Z \to Z'$ induces a homeomorphism of the torus $\bar{\zeta} : \mathbb{T}^d \to \mathbb{T}^d$ so that
\[\bar{\zeta}(U(W,Z)) = U(\zeta(W), \zeta(Z)).\] This amounts to saying that every topological property of $U\paren{W,Z}$ and its embedding in the torus is preserved under $\delta$-perturbations. For example, $U\paren{W,Z}$ contains a periodic path if and only if $U\paren{\zeta\paren{W},Z'}$ does. \new{The set of red points in Figure~\ref{fig:instability} is not $\delta$-good the presence of a giant cycle is unstable to a $\delta$-perturbation.} Given this vocabulary, we can now state our main technical result.

 \begin{theorem}\label{thm:main_technical}
 Let $\epsilon_0 > 0$ and set $\delta_0 = N^{-\epsilon_0}$. Then there exists a coupling of $\paren{Z_1, R_1(p)}$ with $\paren{Z_2, R_2(p + \epsilon)}$ so that there is a $\delta_0$-good subset $W\subset Z_2$ so that $W\subset R_2$ and $$U(R_1, Z_1)\subset U(W, Z_2)\subset U(R_2, Z_2)$$
 with high probability. 
\end{theorem}

\begin{figure}[t]
    \centering
    \includegraphics[width=0.75\linewidth]{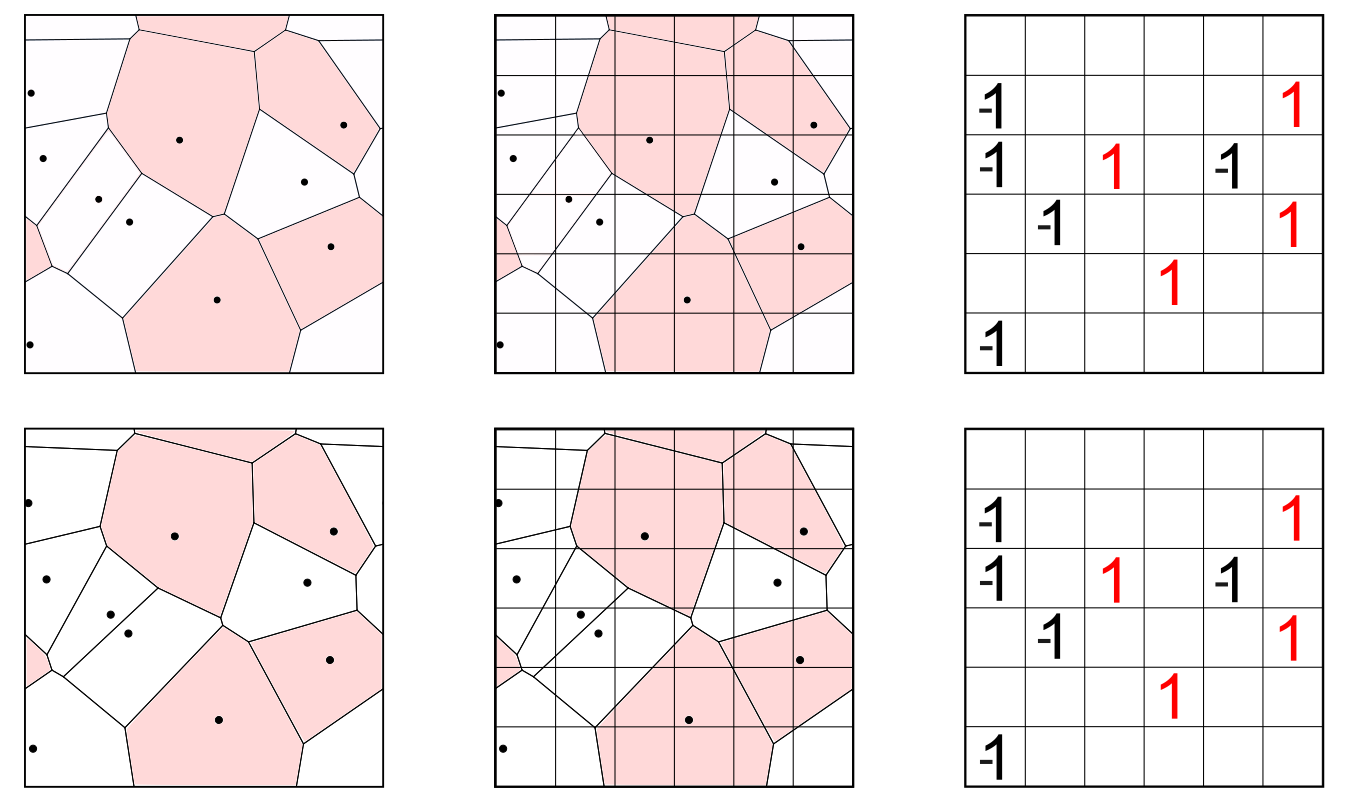}
    \caption{An illustration of Voronoi instability in the two torus. Two Voronoi percolations are shown on the left that are slight perturbations of each other. They share the same discretization, shown on the right, yet only one of them has a giant cycle. %The image on the top shows how a discretization can lose giant cycles.
    }
    \label{fig:instability}
\end{figure}

This allows us to approximate Voronoi percolation with a discrete process that retains the increasing topological properties of $R$ with high probability. In tandem with the sharp threshold theorem of Friedgut and Kalai generalized by Bollobás and Riordan, Theorem \ref{theorem:friedgutkalai}, we demonstrate and determine the sharp threshold for $i$-dimensional homological percolation on a $2i$-dimensional torus. 

Quickly recall that homology depends on an abelian group of coefficients, which we take to be a field $\mathbb{F}$. We will review homology in Section~\ref{section:homologicalpreliminaries} below.
For technical reasons involving representation theory we only consider coefficients from fields with characteristic not equal to $2$. A full explanation is provided in~\cite{duncan2025homological}, but the idea is that if, say, $d=2,$ $\mathbb{F}$ has characteristic $2$, and $A'$ is the event that there is a periodic path which crosses the torus once left-to-right and once bottom-to-top, then event $S$ cannot be ``spun up'' from $A'$ by taking the intersection events symmetric to $A'.$

\begin{theorem}\label{thm:homologicaltransition}
    Suppose $\mathrm{char}(\mathbb{F}) \not= 2$. If $d = 2i$ then
    \[\begin{cases}
        \mathbb{P}_p(A) \rightarrow 0 & p < \frac{1}{2} \\
        \mathbb{P}_p(S) \rightarrow 1 & p > \frac{1}{2}
    \end{cases}\]
    as $N \rightarrow \infty$.
\end{theorem}

\section{Background}
\subsection{Delaunay and Voronoi}
Let $Z$ be a \textbf{Poisson point process}. For $z \in Z$, the \textbf{Voronoi cell} of $z$, $V(z)$, is
\[V(z) = \{x \in \mathbb{T}_N^d : d(x,z) \leq d(x, z'), \forall z' \in Z\}.\]
The collection of Voronoi cells forms a \textbf{polyhedral complex} called a \textbf{Poisson-Voronoi mosaic} that we denote $V(Z)$ and is depicted in Figure~\ref{fig:mosaic}. Note that a Voronoi tesselation on the torus may fail to be a polyhedral complex if the cells are large relative to the diameter of the ambient space. This happens with probability $o(1)$ in the asymptotic regimes we consider. We give a definition of a polyhedral complex below, borrowed from Definition 2.38 in \cite{kozlov2008combinatorial}.

\begin{definition}[Polytope and Polyhedral complex]
    A (convex) \textbf{$n$-polytope} is an $n$-dimensional convex hull of a finite set of points. A \textbf{polyhedral complex} is a collection $X$ of convex polytopes in $\mathbb{R}^n$ such that
    \begin{enumerate}
        \item If $\sigma \in X$, then all faces of $\sigma$ are in $X$.
        \item If $\sigma_1, \sigma_2 \in X$ and $\sigma_1 \cap \sigma_2 \not= \varnothing$, then $\sigma_1 \cap \sigma_2 \in X$.
    \end{enumerate}
\end{definition}

A set of points in a $d$-dimensional Euclidean space is in \textbf{general position} if no $\paren{d+2}$-points lie on on the same $(d-1)$-sphere. We will assume this hypothesis as it occurs with probability one for Poisson point processes. It implies that the intersection of $k$ Voronoi cells is either empty or a $\paren{d-k+1}$-face of the Voronoi tesselation, and all $(d-k+1)$-faces occur as such an intersection (on the torus, we also need to assume that the Voronoi cells are small relative to the diameter of the ambient space). In particular, the intersection of $(d+2)$ Voronoi cells is necessarily empty. In three dimensions, two Voronoi cells can intersect in a $2$-dimensional face, three Voronoi cells can intersect in a $1$-dimensional edge, and four Voronoi cells can intersect in a $0$-dimensional vertex. Note that the points of the integer lattice $\Z^d$ are not in general position. The Voronoi diagram of $\Z^2$ consists of unit squares, some of which meet only in a single vertex. This causes topological complications: the complement of a union of these squares is not necessarily topologically equivalent to the union of complementary squares. On the other hand, the complement of a set of Voronoi cells in a Poisson point process is homeomorphic to the union of complementary cells with probability one.

\begin{figure}[t]
    \centering
    \includegraphics[width=0.5\linewidth]{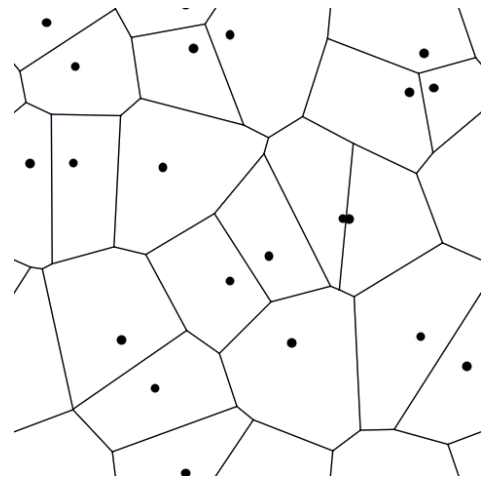}
    \caption{A Poisson-Voronoi mosaic}
    \label{fig:mosaic}
\end{figure}
 The general position hypothesis also allows the definition of a dual cell complex, called the Delaunay triangulation. Before discussing this in more detail, we define a simplicial complex and its terminology.

\begin{definition}[Simplex and simplicial complex]
    An \textbf{$n$-simplex} is the convex hull of $(n+1)$-vertices that are affinely independent. A \textbf{simplicial complex} $\mathcal{K}$ is a collection of simplices such that
    \begin{enumerate}
        \item For all $\sigma \in \mathcal{K}$, every face of $\sigma$ is in $\mathcal{K}$.
        \item For any $\sigma_1, \sigma_2 \in \mathcal{K}$, $\sigma_1 \cap \sigma_2$ is a face of $\sigma_1, \sigma_2$.
    \end{enumerate}
\end{definition}

The Delaunay complex is a natural simplicial complex generated by a point set $Z \subset X$ and a metric on $X$, $\mathrm{dist}$. (In our case, we take $X = \mathbb{T}^d_N$ and $\mathrm{dist}$ as the Euclidean metric.)

\begin{definition}[Delaunay ball]
    A \textbf{Delaunay ball}, or a \textbf{maximal empty ball}, is a ball $B \subset X$ such that $|B\cap Z | = 0$, $|\partial B \cap Z| \geq d +1$.
\end{definition}

If the points are in general position then the Delaunay complex is a simplicial complex called the \textbf{Delaunay triangulation}.

% We assume that the points are in general position, and in particular that exactly $d+1$ points lie on the boundary of any Delaunay ball. This holds with probability one for points sampled from a Poisson point process. As such, to each Delaunay ball we associate a $d$-simplex whose vertices are points on the boundary of $B,$ yielding a simplicial complex called the \textbf{Delaunay triangulation}. 

\begin{definition}[Delaunay triangulation]
    If $Z$ is in general position, then its \textbf{Delaunay triangulation}, $\mathrm{Del}(Z)$, is
    \[\mathrm{Del}(Z) := \{\mathrm{conv}\{z_1,\dots,z_{d+1}\}\mid z_i\in\partial B\; B\text{ is a Delaunay ball}\} \]
    where $\mathrm{conv}\{z_1, \dots, z_{d+1}\}$ denotes the convex hull.  Moreover, we define \linebreak $\mathrm{Del}(W; Z)$ as the subcomplex of $\mathrm{Del}(Z)$ consisting of simplices whose vertices are in $W.$ 
\end{definition}

\begin{figure}[h]
    \centering
    \includegraphics[width=0.4\linewidth]{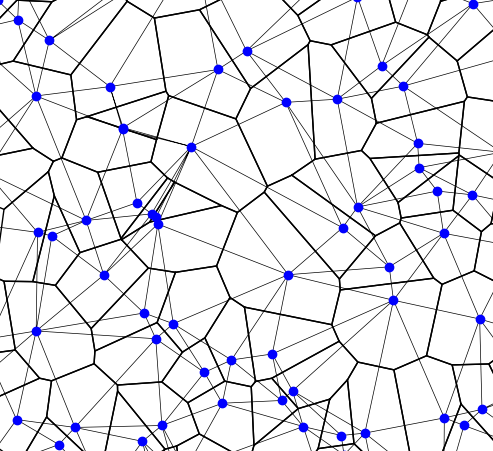}
    \caption{Voronoi diagram with Delaunay complex superimposed.}
    \label{fig:VoronoiDelaunay}
\end{figure}

The duality relationship between Delaunay and Voronoi is given by
\[\mathrm{Del}(Z) = \{\varnothing \not= \sigma \subset Z \, |  \, \bigcap_{u \in \sigma} V(u) \not= \varnothing\}.\]
That is, the Voronoi cells of $z_1,\ldots,z_k$ intersect non-trivially if and only if $\paren{z_1,\ldots,z_k}$ forms a $\paren{k-1}$ simplex of the Delaunay triangulation. For example, in Figure~\ref{fig:VoronoiDelaunay}, the vertices of the Delaunay triangulation are the points $Z,$ two vertices form a Delaunay edge if and only if their Voronoi cells intersect in a Voronoi edge, and three vertices form a Delaunay face if and only if their Voronoi cells intersect in a point. 

This is the higher dimensional generalization of graph duality, $i$-polytopes are dual to $(d-i)$-polytopes. 
More generally, the duality mapping of a polyhedral complex is given by the \textbf{Nerve} (definition borrowed from 10.3 of \cite{Edelsbrunner2014ASC}),
\[\mathrm{Nerve}(X) := \{\varnothing \not= \sigma \subseteq X \, | \, \bigcap \sigma \not= \varnothing \}.\]
One can observe that $\mathrm{Nerve}(V(Z)) = \mathrm{Del}(Z)$ and $\mathrm{Nerve}(\mathrm{Del}(Z)) = V(Z)$. 
Moreover, we introduced the duality relationship to show the following topological relation. 
\begin{lemma}\label{lemma:deformVorDel}
    Let $Z$ be in general position and let $P \subset Z$ finite. Then, $\mathrm{Del}(P ; Z)$ and $U(P ; Z)$ have the same homotopy type.
\end{lemma}

\begin{proof}
    This is an immediate corollary of the Nerve Theorem in \cite{Edelsbrunner2014ASC}
\end{proof}

This lemma allows us to translate topological results on the Delaunay triangulation to the Voronoi setting. Lastly, we introduce the \textbf{star} of a simplex which is shown in Figure~\ref{fig:star}.

\begin{figure}
    \centering
    \includegraphics[width=0.5\linewidth]{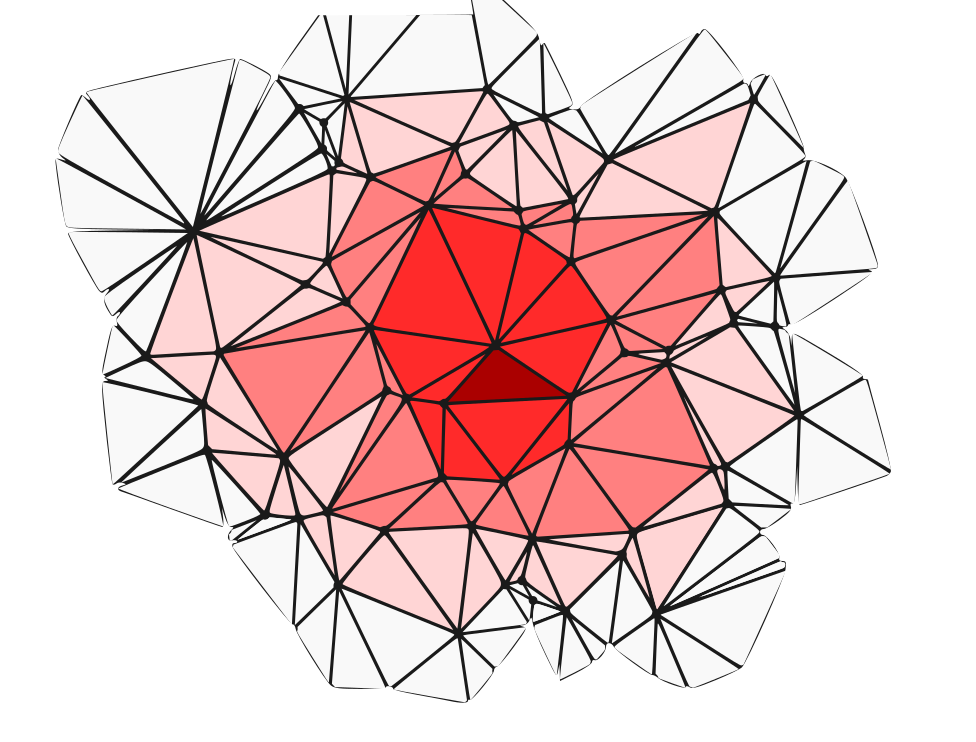}
    \caption{A simplex, its star, star's star, star's star's star, and star's star's star's star shown in progressively lighter colors of red.}
    \label{fig:star}
\end{figure}

\begin{definition}[Star]
The star of a vertex $w\in Z$ is the union of all simplices of $\mathrm{Del}(X)$ containing $w.$ More generally, for $W\subset Z,$  $\mathrm{Star}(W,X)$ is the union of the stars of the vertices of $W.$\end{definition}

\textbf{Notation note:} in this paper, we work solely with polyhedral complexes. More general than polyhedral complexes are cell complexes, which are built of $i$-cells which are homeomorphic to $i$-dimensional balls but are not required to be polytopes. We use terminology from this more general setting: we refer to the faces of the Voronoi tesselation as \textbf{cells}, and the set of $i$-cells as the  \textbf{$i$-skeleton}. 

\subsection{A crash course in homology}
Roughly speaking, homology measures the $i$-dimensional holes of a space. The homology of a topological space can be defined using singular homology, and the topology of a polyhedral complex can be defined using cellular homology \cite{hatcher2002algebraic}, but the definition is simpler and less technical for simplicial complexes. As such, we define homology for simplicial complexes $\mathcal{K}$ and use Lemma \ref{lemma:deformVorDel} to define it for unions of Voronoi cells.

\begin{definition}[$n$-chain]
    Let $C_n(\mathcal{K})$ be the space of formal $\mathbb{F}$-linear combinations of $n$-simplices of $\mathcal{K}$, that is
    \[C_n(\mathcal{K}) = C_n(\mathcal{K}, \mathbb{F}) := \{\sum c_\alpha \sigma_\alpha^n \, : \, c_n \in \mathbb{F}\}.\]
    $C_n(\mathcal{K})$ forms a vector space over $\mathbb{F}$.
\end{definition}

\textbf{The meaning of coefficients.} The simplest choice of coefficients is $\mathbb{F} = \mathbb{Z}_2$. The coefficients, $0, 1$, can be viewed as excluding or including a face, so $C_n(\mathcal{K};\Z_q)$ can be formally identified with the power set of the collection of $n$-faces. A different choice of coefficients is $\mathbb{F} = \mathbb{Z}_3$. The coefficients, $0, 1, 2$, can be viewed as excluding $c_\alpha = 0$ a face or including $c_\alpha = 1,2$ face with orientation (say up or down). As we discuss briefly below, the ranks of the homology groups depend on the coefficients in general.  %More broadly, coefficents from $\mathbb{Z}_n$ implies that $n$ sums of the same boundary vanishes (e.g, a loop wrapping around the torus $n$-times), unlike coefficients from $\mathbb{Q}$. 

We will discuss the meaning of the coefficients after defining homology. First, we introduce the boundary homomorphism.

\begin{definition}[Boundary]
    The \textbf{boundary homomorphism}, $\partial_n : C_n(\mathcal{K}) \to C_{n -1}(\mathcal{K})$ is given by
    \[\partial \sigma^k = \sum_{i=0}^{n} (-1)^i \sigma^k |[v_0, \dots, \hat{v_i}, \dots, v_n].\]
    The term $[v_0, \dots, \hat{v_i}, \dots, v_n]$ indicates the simplex generated by $v_0, \dots, v_n$ excluding $v_i$. Moreover, $B_n(X) := \mathrm{Im} \, \partial_{n+1}$ are the \textbf{$n$-boundaries} and $Z_n(X) := \mathrm{Ker} \, \partial_{n}$ are the \textbf{$n$-cycles}. 
\end{definition}

The boundary map satisfies the fundamental relation $\partial_n \circ \partial_{n - 1} = 0$. E.g, $B_{n}(X) \subset Z_n(X)$. If there exists an $n$-cycle that is not the boundary of an $(n+1)$-chain, e.g, $Z_n(X) \subsetneq B_n(X)$, then we say there exists an $n$-dimensional hole. This is formalized with homology.

\begin{definition}[Homology]
    The $n$th homology is defined as
    \[H_n(\mathcal{K}) = Z_n(X)/B_n(X).\]
\end{definition}

\textbf{Example: }Consider two adjacent triangles as shown in \ref{fig:homologyEx}. (In this example, triangle will refer to the edges that make up a triangle, ignoring the face if one exists).

\begin{figure}[h]
    \centering
    \includegraphics[width=0.5\linewidth]{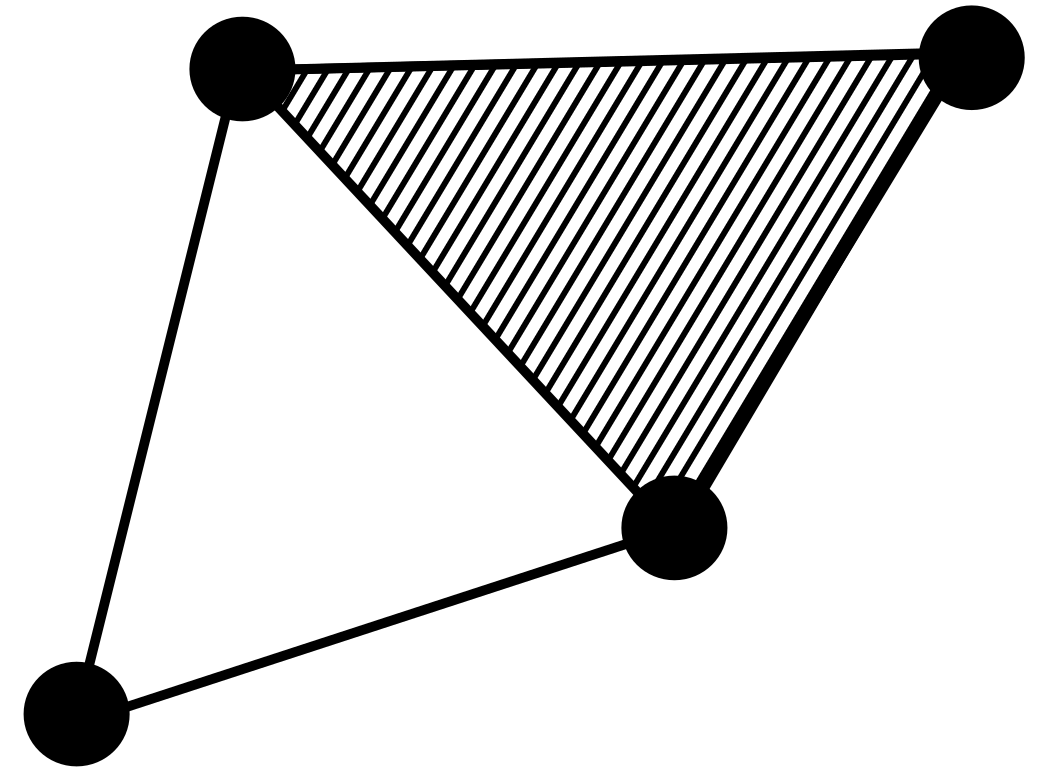}
    \caption{Two adjacent triangles. One is filled in and the other is not.}
    \label{fig:homologyEx}
\end{figure}

There are $4$ vertices, $5$ edges, and $1$ face.
\[C_0 \cong \mathbb{F}^4, \quad C_1 \cong \mathbb{F}^5, \quad C_2 \cong \mathbb{F}^1, \quad C_n \cong 0 \quad \forall n \geq 3.\]
The spaces of $1$-chains and $0$-chains are
\begin{align*}
    Z_1 &= \{\text{Linear combinations of the triangles}\} \cong \mathbb{F}^2, \\
    Z_0 &= \{\text{Linear combinations of the vertices}\} \cong \mathbb{F}^{4}.
\end{align*}
Moreover, the spaces of $1$-boundaries and $0$-boundaries are
\begin{align*}
    B_1 &= \{\text{Linear combinations of the triangle on the right}\} \cong \mathbb{F}, \\
    B_0 &= \{\text{Linear combinations of boundaries of edges}\} \cong \mathbb{F}^3.
\end{align*}
So, our homology groups are:
\begin{align*}
    H_1 &= Z_1 / B_1 \cong \mathbb{F} \\
    H_0 &= Z_0 / B_0 \cong \mathbb{F}.
\end{align*}
The non-triviality of $H_1$ reflects the existence of a hole in simplicial complex. Moreover, $\dim(H_0) = 1$ reflects that the complex is a single connected component.

\textbf{More on coefficients: }When $i=0$ or when $\mathcal{K}$ is embeddable $\R^2$ or $\R^3$ then the rank of the homology does not depend on the choice of coefficients and counts the number of connected components. The higher homology groups depend on the choice of coefficients in general. For example, any non-orientable surface $S$ such as the Klein bottle or real projective plane then $H_2\paren{S;\Z_2}\cong \Z_2$ but $H_2\paren{S;\Z_q}=0$ for odd primes $q$: the $2$-simplices cannot be oriented so that their boundaries cancel out. 

A continuous function $f:X\to Y$ induces a linear transformation on homology $f_*:H_n(X)\to H_n(Y).$ This is particularly easy to define for an inclusion map $\iota:X\hookrightarrow Y$ when $X\subset Y.$ A cycle in $X$ is automatically a cycle in $Y$ and $B_n(X)$ is a subspace of $B_n(Y),$ so $\iota_*$ sends a coset  $z + B_n(X)$ to the potentially larger coset $z + B_n(Y).$

A key result is that homology is a homotopy invariant.
\begin{proposition}[Corollary 2.11 of \cite{hatcher2002algebraic}]
    The maps $f_* : H_n(X) \to H_n(Y)$ induced by a homotopy equivalence $f : X \to Y$ are isomorphisms for all $n$.
\end{proposition}

\begin{corollary}
    Let $P = U(Q, Z)$ and let $I = \mathrm{Del}(Q; Z)$. Then,
    $H_n(P) \cong H_n(I)$.
\end{corollary}

\begin{proof}
    By Lemma \ref{lemma:deformVorDel}, $I$ is homotopy equivalent to $P$. By the previous corollary, the rest follows.
\end{proof}

\subsection{Voronoi Percolation}
\textbf{Voronoi percolation} is defined by including the Voronoi cells of a Poisson point process $Z$ independently with probability $p.$ More formally, the underlying probability space it is a pair $\paren{Z,R}$ where we include each point of $Z$ in $R$ independently with probability $p.$ See Figure~\ref{figure:VoronoiPercolation}.

In \cite{Bollob_s_2005}, Bollobás and Riordan proved for $d = 2$, the critical probability is $1/2$. A key technical lemma is the following coupling result, adapted to our terminology

\begin{theorem}[Bollobás and Riordan, 2005]
    Let $(Z_1,R_1)$ be a Voronoi percolation on $\mathbb{T}_N^2$. For $\epsilon > 0$ and $\delta \leq N^{-\epsilon}$, there exists a coupling of $(Z,R(p))$ with $(Z_2,R_2(p+\epsilon))$ such that, for every path $\gamma$ in $U(R_1,Z_1)$ there exists a path $\gamma'$ in $U(R_2, Z_2)$ and every point of $\gamma'$ is within distance $(\log{N})^2$ of some point of $\gamma$ and vice-versa so that every point $4\delta$-close to $\gamma'$ is red.
\end{theorem}

Theorem \ref{thm:main_technical} generalizes this statement to $d$-dimensions and strengthens it to control all increasing topological properties.

%Namely, we strengthen ``red path preserving and stabilizing" to ``increasing topological properties preserving and stabilizing". Furthermore, we establishing the coupling for Voronoi percolation on $\T_N^d$.

\subsection{Homological Percolation and Discretization}\label{section:homologicalpreliminaries}
Let $P \subset \T_N^d$. The inclusion map $P \hookrightarrow \mathbb{T}_N^d$ induces a map on the homology $\phi_i : H_i(P, \mathbb{F}) \to H_i(\mathbb{T}^d_N, \mathbb{F})$. We define our homological percolation events below.

\begin{definition}
    Let $A := \{\mathrm{rank} \, \phi_i > 0\}, \quad S := \{\mathrm{nullity} \, \phi_i = 0\}.$
\end{definition}

Our coupling result, Theorem \ref{thm:main_technical}, states that slightly increasing the probability and discretizing the Voronoi percolation preserves increasing events invariant under homeomorphisms of $\mathbb{T}^d$, including $A$ and $S$. With the coupling in mind, we define our formal discretization here, which is the same as the crude state discretization constructed in \cite{Bollob_s_2005}.

\begin{definition}
    Let $\delta = N^{-\epsilon}$. We partition $\mathbb{T}^d$ into $\lceil{(N/\delta)^d}\rceil$ cubes of length $\delta$. Each cube is assigned a value in $\{-1,0,1\}$,
    \begin{enumerate}
        \item $-1$ if it contains a white point 
        \item $0$ if it contains no points 
        \item $1$ if it contains only red points.
    \end{enumerate}
    The assignment of $(Z, R)$ is called the \textbf{coarse state}, denoted $\mathrm{Coarse}_\delta(Z,R).$
\end{definition}
%With Theorem \ref{thm:main_technical} in mind, we call it a \textbf{coarse state} rather than a crude state as the coupling shows the giant cycles are probabilistically persistent under a small increase to $p$. 
Later, we study the probability space of the coarse state to apply Theorem \ref{theorem:friedgutkalai}. To make events well-defined on the coarse state, we introduce the notion of a \textbf{stable event}.

\begin{definition}[Stable event]
    For any event $X$ define
    \[X_{\text{stable}} := \{V \, | \, \text{$X$ occurs on any $V'$ with the same coarse state as $V$}\}.\]
\end{definition}

Events on percolation processes are usually defined in terms of the geometry of the percolation process itself. Such examples include: there exists an infinite open component containing the origin, there exists $k$ disjoint open paths from the origin to the boundary of a box of length $n$ (called a $k$-arm event), and so on. A stable event is different, instead it is an invariance under the placement of points into their respective cubes. As $X_{\mathrm{stable}} \subset X$, we obtain the useful inequality 
\[\mathbb{P}_p(X_\mathrm{stable}) \leq \mathbb{P}_p(X).\]

\subsection{Probabilistic tools}
Foundational to our proof is the sharp threshold theorem of Friedgut and Kalai \cite{Friedgut1996EveryMG}, stated in a form generalized by Bollobás and Riordan \cite{Bollob_s_2005} to a probability space over $\{-1, 0, 1\}^n$ (rather than $\{0,1\}^n$).

\begin{theorem}[Bollobás and Riordan]\label{theorem:friedgutkalai}
    Let $X$ be an increasing event on \linebreak $\{-1,0,1\}^{n}$ with some probability measure $\mathbb{P}_{p_1,p_2}^n$ that is invariant under a transitive group action and
    \begin{equation}\label{equation:fkinequality1}
        0 < q_{1} < p_{1} < 1/e, \quad 0 < p_{2} < q_{2} < 1/e\,.
    \end{equation}
    If $\mathbb{P}_{p_1,p_{2}}^{n}(X) > \eta$, then $\mathbb{P}_{q_1,q_2}^n(X) > 1 - \eta$ whenever
    \begin{equation}\label{equation:fkinequality2}
        \min\{q_{2} - p_{2}, p_{1} - q_{1}\} \geq c \log{(1/\eta)} p_{max}\log{(1/p_{max})/\log{n}}
    \end{equation}
    where $p_{max} = \max\{p_1, q_2\}.$
\end{theorem}

\subsection{Paper outline}
The majority of the paper is focused on proving Theorem \ref{thm:main_technical}. We develop key results, Lemma \ref{lemma:clusterBound} and Theorem \ref{theorem:perturbwrap} in sections \ref{section:badasymptotics} and \ref{section:voronoitopology} respectively. Then, in section \ref{section:coupling}, we demonstrate Theorem \ref{thm:main_technical} by carefully constructing a coupling that yields an appropriate discretization of a Voronoi percolation with high probability. Lastly, in section \ref{section:evendimension}, with our discretization in hand and results from \cite{duncan2025homological}, we prove Theorem~\ref{thm:homologicaltransition} on homological percolation.

As mentioned, many of these arguments, especially in Sections \ref{section:badasymptotics}, \ref{section:coupling}, are modifications of Bollobás and Riordan's seminal work on Voronoi percolation.

\section{Asymptotics of bad points}\label{section:badasymptotics}
\subsection{Delaunay Triangulations}
Crucially, \cite{boissonnat2013stability} established several theorems on the stability of Delaunay triangulations under perturbations of both the point locations and the metric. These results imply corresponding stability theorems for Voronoi tessellations. We review the the terminology and results of \cite{boissonnat2013stability} and our main result of the section.

\new{In a nutshell, notions of instability for Delaunay triangulations indicate how close points are to not being in general position. That is, they measure how close a Delaunay ball is to a point in its exterior, how close a tuple of $(d+1)$ points is to a plane, or how close a pair of points are to each other. A Poisson point process is in general position with probability $1$; in this section we study the asymptotics of points that are ``nearly'' out of general position.

%\begin{figure}[htbp]
\begin{figure}[ht]
  \centering
  \begin{minipage}{0.4\textwidth}
    \centering
    \includegraphics[width=\textwidth]{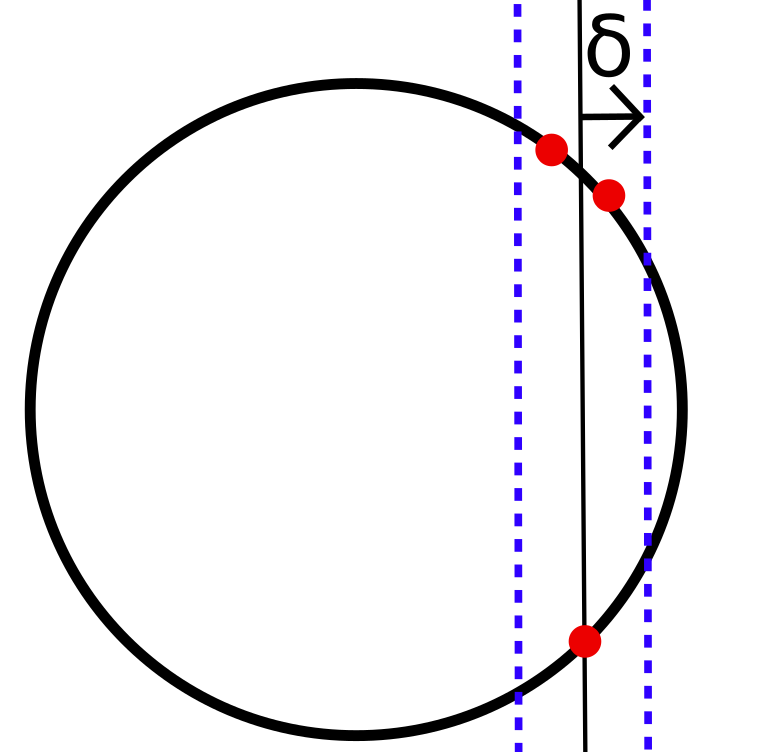}
  \end{minipage}\hfill 
  \begin{minipage}{0.5\textwidth}
    \centering
    \includegraphics[width=\textwidth]{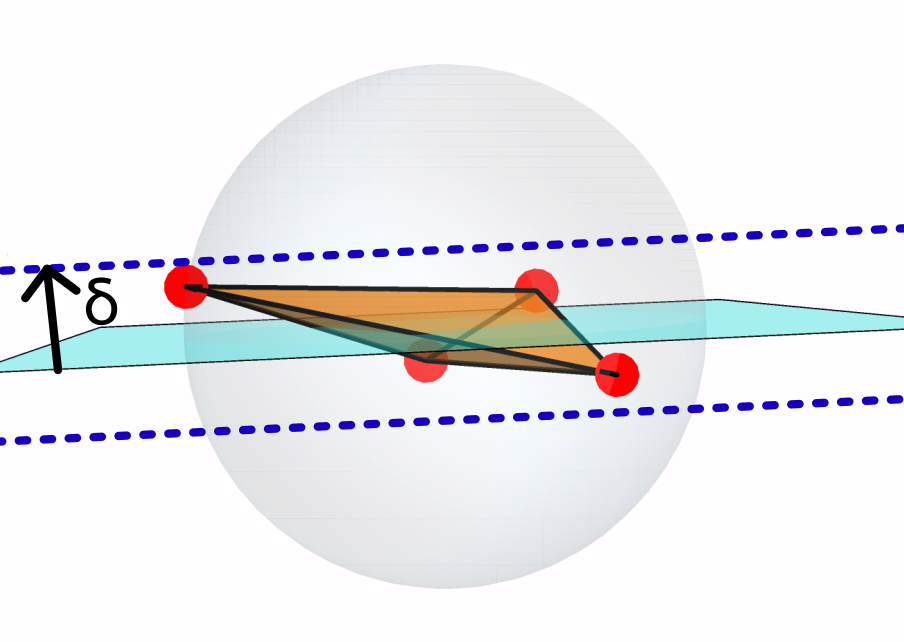}
  \end{minipage}
  \caption{A $\delta$-unstable $(1+2)$-tuple and a $\delta$-unstable $(2+2)$-tuple. On the left, a $(1+2)$-coplanar tuple forces a $\delta$-close pair. On the right, a $(2+2)$-coplanar tuple need not form a $\delta$-close pair.}
  \label{fig:coplanarity}
\end{figure}

These asymptotics for planar Poisson point processes were considered in \cite{Bollob_s_2005}. However, a new type of instability arises in higher dimensions.  In two dimensions, three nearly coplanar points \new{of the same simplex} must necessarily include a close pair. However, in higher dimensions, $(d+1)$ points can be approximately coplanar without the presence of a close pair. \new{See Figure~\ref{fig:coplanarity} for a comparison}. Our first definition is taken from \cite{boissonnat2013stability} and measures the sparseness of a point sample and the instability of a simplex.}
 
\begin{definition}[$R$-sample set]
    A point set $Z\subset \T^d_N$ is an \textbf{$R$-sample set} if the distance from any point $x \in \mathbb{T}^d_N$ to $Z$ is at most $R$.
\end{definition}

% A simplex formed by $d+1$ points is $\delta$-unstable if there is another point that is close to its Delaunay ball. 

\begin{definition}[$\delta$-unstable simplex]
    Let $\sigma$ be a simplex with $(d+1)$-points and a Delaunay ball $B$. We say $\sigma$ is \textbf{$\delta$-unstable} if 
    \[\mathrm{dist}_{\mathbb{T}_N^d}(q, \partial B) > \delta\]
    for some $q \in P \setminus \sigma$.
\end{definition}

\new{It will be useful to consider a broader class of unstable $(d+2)$-tuples. }

\begin{definition}[$\delta$-unstable tuple]
    A $(d+2)$-tuple of points $\{z_1, \dots, z_{d+2}\}$ is \textbf{$\delta$-unstable} if there exists a spherical shell of inner radius $r < \log{N}$ and outer radius $r + \delta$ that contains those $(d+2)$-points.
\end{definition}

\new{A $\delta$-unstable $(d+2)$-tuple contains a $\delta$-unstable simplex with circumradius at most $\log{N} + \delta$ for some $Z' \subset Z.$ We consider several sub-cases.}

\begin{definition}[$\delta$-close pair, $\delta$-coplanar tuple, $\delta$-unstable separated tuple]    
    A pair of points $\{z_1, z_2\}$ is a \textbf{$\delta$-close pair} if the distance between them is less than $\delta$. \new{A $(k+2)$-tuple of points $\{z_1, \dots, z_{k+1}\}$, where $0\leq k < d,$ is \textbf{$\delta$-coplanar} if it has diameter at most ${\log{N}}$, belongs to a thickened $k$-dimensional plane with width $\delta$, and does not contain a $\delta^{1/2^{k-k'}}$-coplanar $(k'+2)$-tuple for any $0\leq k' < k$. A $\delta$-unstable $(d+2)$-tuple is called \textbf{separated} if it does not contain a $\delta^{1/2}$-coplanar $(k+2)$-tuple for any $0\leq k < d.$ Let $Z_\delta$ denote the set of all points in $\delta$-close pairs, $\delta$-coplanar tuples, and $\delta$-unstable separated tuples. We call points of $Z_\delta$ \textbf{$\delta$-unstable}.}
\end{definition}

\begin{figure}[h]
    \centering
    \includegraphics[width=0.75\linewidth]{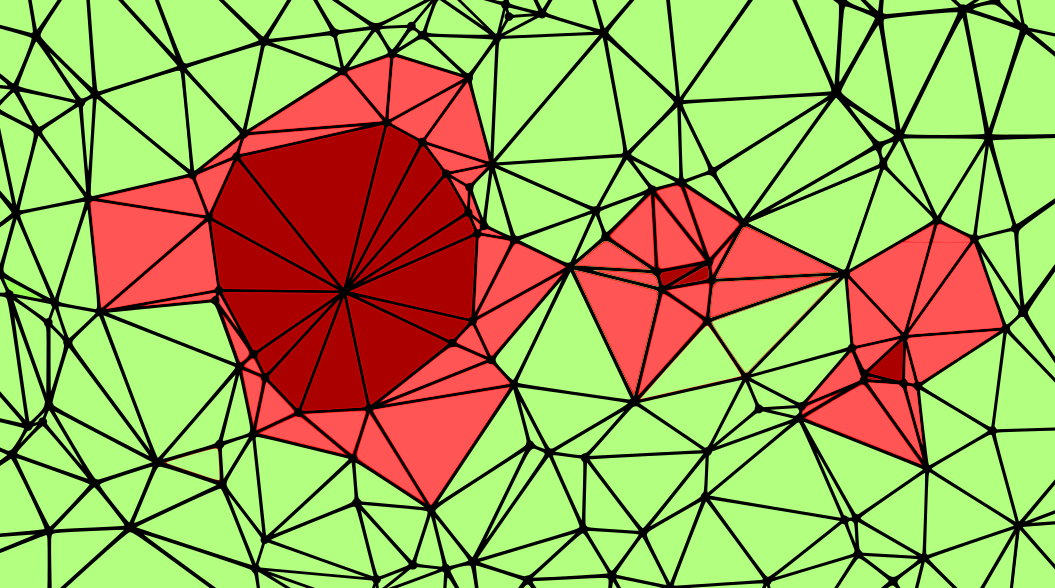}
    \caption{A rough depiction of the Delaunay triangulation of a point set with a $\delta$-bad cluster. In red are $\delta$-unstable simplices, in light red are $\delta$-bad simplices, and in light green are $\delta$-good simplices. On the left is a $\delta$-close pair, and on the right are $\delta$-unstable separated tuples.}
    \label{fig:deltabadcluster}
\end{figure}

\new{We regard a $0$-plane as a point so a $\delta$-coplanar $(0+2)$-tuple is a $\delta$-close pair. In this section, our calculation for the asymptotics of a $\delta$-close pairs is different than $\delta$-coplanar tuples, so we handle them explicitly. However, in other sections, we will consider them as a $(0+2)$-subtuple of a $\delta$-coplanar $(k+2)$-tuple. This distinction is significant, as it determines a much larger magnitude of $\delta$ on the level of the subtuple. Despite this, $\delta$-close pairs serve as a prototypical example of instability in the Delaunay complex, so we will mention them in discussions of stability.

We also count neighbors of $\delta$-unstable points. We call a point \emph{bad} if it is $\delta$-unstable or if it shares a Delaunay edge with a $\delta$-unstable point. The situation is depicted in Figure~\ref{fig:deltabadcluster}.}

\begin{definition}[$\delta$-bad point and $\delta$-bad cluster]
\label{defn:cluster}
    Let $P \subset Z$. We say that a point $z\in Z$ is \textbf{$\delta$-bad} if it is contained in the star of a $\delta$-unstable point. A \textbf{$\delta$-bad cluster} is a maximal component of $\delta$-bad points. 
\end{definition}

%\new{Although our definitions of $\delta$-coplanar and $\delta$-unstable separated tuple have much stricter technical conditions than a more straightforward definition, the $\delta$-bad cluster contains all the $\delta$-unstable tuples that are simplices in the Delaunay complex, and it helps expedite some Lemmas in this section. As we will see in the next section, the unstable points are wholly responsible for instability in the Delaunay complex.}

For the rest of this section, just as in \cite{Bollob_s_2005}, we do not explicitly discuss the underlying Delaunay triangulation. Instead, we only discuss unstable and bad points.

Lastly, we introduce the following notation to streamline the statements of some of our technical results. When we write that a random variable $X$ is $o(f(N))$ (resp. $O(f(N))$) with probability $1-o(1),$ we mean that for any $\eta > 0$ (resp. there exists $\eta>0$ so that), $X<\eta f(N)$ with high probability as $N\to \infty.$ Our goal in this section is to prove the following lemma, which corresponds to to Lemma 6.4 of \cite{Bollob_s_2005}. 

\begin{lemma}\label{lemma:clusterBound}
    Every $\delta$-bad cluster $C$ has diameter at most $\log{N}$ with probability $1 - o(1)$. Moreover, every $\delta$-bad cluster $C$ has size $o(\log{N})$ and contains $O\paren{1}$ unstable points with probability $1 - o(1)$.
\end{lemma}

By ``diameter" we mean the longest distance between any two points of the cluster. In \cite{Bollob_s_2005}, they enlist a technical lemma (Lemma 19) to obtain a bound on the maximal degree of $\delta$-bad points with high probability. In explaining the technical complications, they give a quick argument for the existence of a $\delta$-unstable point with degree $\Theta\paren{\frac{\log{N}}{\log{\log{N}}}}$. It turns out that this is asymptotically sharp.

\begin{theorem}[Bonnet, Chenavier~\cite{bonnet2018maximaldegreepoissondelaunaygraph}]\label{theorem:maximumdegree}
    Let $\Delta_N$ be the maximal degree in a Poisson-Delaunay triangulation over all vertices in $N^{1/d}[0,1]^d$, $d \geq 2$. Then there exists a deterministic function $N \mapsto J_{N}$, $N > 0$, with values in $\mathbb{N}$, such that
    \begin{enumerate}[label = (\roman*)]
        \item Let $l_d = \lfloor \frac{d+3}{2} \rfloor$,
        $$\mathbb{P}(\Delta_N \in \{J_{N}, J_{N} + 1, \dots, J_{N} + l_{d}\}) \underset{N \rightarrow \infty}{\longrightarrow} 1;$$
        \item 
        $$J_{N} \underset{N \rightarrow \infty}{\sim} \frac{d - 1}{2} \frac{\log{N}}{\log{\log{N}}}$$.
    \end{enumerate}
\end{theorem}

We don't need the full strength of this theorem so we write the following corollary for convenience.

\begin{corollary}{\label{corollary:logarithmBound}}
     The maximal degree of the Poisson point process on $\mathbb{T}_N^d$ is $o(\log{N})$ with probability $1 - o(1)$.
\end{corollary}

\begin{proof}
    Let $\Delta_{N^d}$ be the maximal degree in a Poisson-Delaunay graph over all nodes in $N[0,1]^d = (N^d)^{1/d} [0,1]^d$. By the previous theorem,
    \[\mathbb{P}\bigg(\Delta_{N^d} = O\bigg(\frac{d-1}{2}\frac{\log{(N^d)}}{\log{\log{(N^d)}}}\bigg)\bigg) \to 1\,.\]
    Lastly, $\frac{d - 1}{2} \frac{\log{(N^d)}}{\log{\log{(N^d)}}} = o(\log{N}).$ 
\end{proof}

\subsection{Poisson-Delaunay preliminaries}
We first review some elementary results on Poisson point processes. Throughout this section, recall that $\delta = N^{-\epsilon}$. Moreover, for a Borel set $\Omega \subset \T_N^d$, we write $\|\Omega\|$ to denote its volume.

\begin{lemma}{\label{lemma:emptyDisc}}
    Every simplex $\sigma$ has circumradius at most $\rho \sqrt[d]{\log{N}}$ for some $\rho > 0$ with probability $1 - o(1)$. In particular, $Z$ is a $(\log{N})$-sample set with probability $1 - o(1)$.
\end{lemma}

\begin{proof}
    For any Borel set $\Omega$ and number of points $\xi := |\Omega \cap Z|$, we have
    \[\mathbb{P}(\xi = 0) = e^{-\|\Omega\|}.\]
    If $\Omega$ is a ball of radius $\rho \sqrt[d]{\log{N}}$ and $\rho > \sqrt[d]{d}$, then $\|\Omega\| > d\log{N}$. Hence,
    \[\mathbb{P}(\xi = 0) = e^{-\rho^d\log{N}} = N^{-\rho^d} = o(N^{-d}).\]
    Moreover, the torus may be covered with $O(N^d)$ overlapping balls $\Omega$. The probability that at least one ball is empty is $o(1)$. Hence, the probability that none are empty is $1 - o(1)$. 
\end{proof}

Let $\Lambda$ be a cube of length $\log{N}$ and let $|\Lambda|$ denote the number of vertices of $Z$ in the cube, so $|\Lambda|$ is a Poisson random variable with mean $(\log{N})^{d}$. 

\begin{lemma}\label{lemma:chernoff}
    Let $B_1(\Lambda)$ be the event that $\Lambda$ contains more than $3(\log{N})^{d}$ points. Then, $\mathbb{P}(B_{1}) = o(N^{-d})$.
\end{lemma}

\begin{proof}
    By using a Chernoff bound, we have
    \begin{align*}
        \mathbb{P}(|\Lambda| \geq a) &\leq \bigg(\frac{e\|\Lambda\|}{a}\bigg)^{a} e^{-\|\Lambda\|}\,. \\
        \intertext{Setting $a = 3(\log{N})^{d}$ and $\|\Lambda\| = (\log{N})^{d}$, we obtain}
         \mathbb{P}(|\Lambda| \geq 3) &\leq \bigg(\frac{e(\log{N})^{d}}{3(\log{N})^{d}}\bigg)^{3(\log{N})^{d}} e^{-(\log{N})^{d}} \\
        &= \bigg(\frac{e}{3}\bigg)^{3(\log{N})^{d}} e^{-(\log{N})^{d}} \\
        &= o(N^{-d})\,.
    \end{align*}
\end{proof}

\subsection{Unstable points and bad clusters}
We adapt the proof of Lemma 6.4 of \cite{Bollob_s_2005} to higher dimensions, and account for the phenomenon of $\delta$-coplanar tuples which do not occur in two dimensions. We also make use of recent results such as as the Maximum Degree Theorem of \cite{bonnet2018maximaldegreepoissondelaunaygraph} for greater brevity.

% Throughout this subsection, we will assume $N$ is sufficiently large enough, so that
% $$\delta^{-\alpha} \gg \delta^{-1} \gg \delta^{\alpha^{-1}} \gg (\log{N})^{\beta} \gg c.$$ Where $0 < \alpha < 1,$ $\beta > 0,$ and $c \geq 0$ some fixed constant.

%Most results of this section are a \textit{mutatis muntandis} adaptation of Lemma 6.4 of \cite{Bollob_s_2005}. That is, we follow the same line of arguments, but extend them to higher dimensions and employ recent results such as the Maximum Degree Theorem of \cite{bonnet2018maximaldegreepoissondelaunaygraph} for greater brevity.

\begin{lemma}\label{lemma:pairBound}
    Let $\Omega$ be a Borel set satisfying $\|\Omega\| \geq 1$. Place $\xi =O((\log{N})^d)$ points i.i.d. uniformly in $\Omega$ and let $M > 0$. The probability that there are more than $\frac{M (d+1)}{d \epsilon}$ $\delta$-close pairs is $O(\delta^{M/2}).$
\end{lemma}

\begin{proof}
    Let $z_{1}, \dots, z_{\xi}$ be an ordering of the points and let $B_{\delta}(z)$ denote a $\delta$-ball centered at a point $z$. The conditional probability that $z_i$ is $\delta$-close to an earlier point $z_{1}, \dots, z_{i - 1}$ is at most $$\frac{(i-1)\|B_{\delta}\|}{\|\Omega\|} \leq \xi\|B_\delta\| = O(\xi\delta)=O(\delta (\log{N})^{d}) \leq O(\delta^{1/2})\,,$$ where we have used $\|\Omega\| \geq 1$, $\xi= O((\log{N})^{d})$, and $\|B_\delta\| = O(\delta)$. So, the probability that $M$ points $z_i$ are $\delta$-close to a previously placed point is $O(\delta^{M/2})$. 
    
    Each $z_i$ could be $\delta$-close to many earlier points, however, with probability $1 - o(N^{-d})$, there are at most $(d+1)/d\epsilon$ possibly earlier points. Indeed, the probability of placing more than $a = (d+1)/d\epsilon$ points in $B_\delta(z_i)$ is: 
    \begin{align*}
        &\mathbb{P}(B_\delta(z_i) > a \, | \, \xi) \\ &= \sum_{k={a+1}}^\xi\binom{\xi}{k} \bigg(\frac{\|B_\delta(z_i)\|}{\|\Omega\|}\bigg)^k\bigg(1-\frac{\|B_\delta(z_i)\|}{\|\Omega\|}\bigg)^{\xi-k} \\
        &\leq \bigg(\frac{\|B_{\delta}(z_i)\|}{\|\Omega\|}\bigg) ^a  \underbrace{\sum_{k={a+1}}^\xi\binom{\xi}{k} \bigg(\frac{\|B_\delta(z_i)\|}{\|\Omega\|}\bigg)^{k-a}\bigg(1-\frac{\|B_\delta(z_i)\|}{\|\Omega\|}\bigg)^{\xi-k}}_{\leq1} \\
        &\leq \|B_\delta(z_i)\|^a \\
        &= O(\delta^a) \\
        &= o(N^{-d}).
    \end{align*}

    Thus, the conditional probability that $\Omega$ has more than $M (d+1)/d\epsilon$ points given $\xi$ and an ordering of points is $O(\delta^{M/2})$. As the bound on the conditional probability is independent of $\xi$ and the ordering of points, we infer that the probability that $\Omega$ contains more than $M (d+1)/d\epsilon $ points that are $\delta$-close to an earlier point is $O((\delta^{1/2})^{M}).$ 
\end{proof}

% In particular, we are concerned about the case where $\Omega$ is a cube of length $\log{N}$, which we wrote as $\Lambda$ in the last subsection.

% \begin{corollary}
%     Let $B_2(\Lambda)$ be the event that $\Lambda$ has more than $\frac{2(d+1)^2}{d\epsilon^2}$ points in $\delta$-close pairs. Then, $\mathbb{P}(B_2) = o(N^{-d})$.
% \end{corollary}

% \begin{proof}
%     We apply the previous lemma using $M = 2(d+1)/\epsilon$. Firstly, the conditional probability of $B_2$ given that $\Lambda$ contains fewer than $3(\log{N})^d$ points is:
%     \[\mathbb{P}(B_2 \, | \, \xi \leq 3(\log{N})^d) = O(\delta^{(d+1)/\epsilon}) = o(N^{-d}).\]
%     By Lemma \ref{lemma:chernoff}, $\mathbb{P}(B_1) = o(N^{-d})$. Thus,
%     \[\mathbb{P}(B_2) = \mathbb{P}(B_2 \, | B_1) \mathbb{P}(B_1) + \mathbb{P}(B_2 \, | B_1^c) \mathbb{P}(B_1^c) = o(N^{-d}).\]
%     % so $\Lambda$ has $O(\log{N})$ points. It follows that $\mathbb{P}(B_2 \, | \, B_1^c) = O(\delta^{(d+1)/\epsilon}) = o(N^{-d})$.
% \end{proof}

\new{We proceed similarly for $\delta$-coplanar $(k+2)$-tuples, but first enlist the use of a technical lemma.}

\begin{lemma}
The set of $z$ such that $z_1, \dots, z_{k+1}, z \subset \R^d$ forms a $\delta$-coplanar $(k+2)$-tuple has volume $o(\delta^{1/4})$ uniformly in $z_1,\ldots,z_{k+1}.$ 
\end{lemma}

\begin{proof}
    Suppose $z_1,\ldots,z_k,z_{k+1}$ form a $\delta$-coplanar $(k+2)$-tuple and let $H$ be a $k$-plane satisfying $\dist(H, z_i) < \delta$ for all $i.$ Let $H_0$ be the plane generated by $z_1, \ldots, z_{k+1}.$ If $H_0$ and $H$ are parallel, then we are done. Otherwise, without loss of generality, by translation and rotation, we may assume that $z_1, \dots, z_{k+1},z \subset \R^{k+1}=\R^{k+1}\times\set{0}^{d-k-1},$ $H_0$ is the solution set to the equation $x_{k+1} = 0$ in $\R^{k+1}$, and $0 \in H$. 
    Let us write $H$ as the solution set of a linear system,
    $$H = \set{x \in \R^{k+1} \mid \sum_{i=1}^{k+1} b_i x_i = 0},$$
    where $b_1, \ldots, b_{k+1} \in \R.$ Since $z_1, \ldots, z_{k+1} \in H$ are not $\delta^{1/2}$-coplanar, $b_{k+1} \not= 0$ and we may view $H$ as the graph of a height function in $\R^{k+1}.$  We parameterize $H$ by $x_{k+1} = \sum_{i = 1}^{k} \frac{b_i}{-b_{k}} x_i =: f(x).$ For $x \in H_0,$ we call $f(x)$ the height of $x.$ With our assumptions finished, we argue that any point $x \in H$ that is $\log{N}$-close to $z_1$ has height $O(\delta^{1/3}).$

    Let $\proj_H(x)$ and $F\paren{x}$ be the projections onto $H$ and $H_0,$ respectively, and set $w_i=F\paren{\proj_H(z_i)}.$ 
    We have by Pythagorean theorem,
    $$\dist(\proj_H(z_i), w_i)^2 + \dist(w_i, z_i)^2 = \dist\paren{z_i, \proj_H\paren{z_i}}^2 \leq \delta^2$$
    so
    $$\dist\paren{w_i, z_i} \leq \delta\,.$$

    The tuple $z_1, \dots, z_{k+1}$ is not $\delta^{1/2}$-coplanar. That is, for $0 \leq k' \leq k-1,$ every $(k'+2)$-subtuple of $z_1, \dots, z_{k+1}$ is not $\delta^{1/2^{k-k'+1}}$-close to a $k'$-plane. Let $I$ be the index set of a $(k'+2)$-subtuple, then $\set{w_i}_{i \in I}$ is not $(\delta^{1/2^{k-k'+1}} - \delta)$-close to a $k'$-plane. Otherwise, if $H_w$ is a $k'$-plane where $\dist(w_i, H_w) < \delta^{1/2^{k-k'}}$, then $\dist(z_i, H_w) \leq \dist(z_i, w_i) + \dist(w_i, H_w) \leq \delta^{1/2^{k-k'}},$ a contradiction. Lastly, when $\delta$ is sufficiently small we have the inequality $(\delta^{1/2^{k-k'}} - \delta) \geq C(\delta^{1/2} - \delta)^{1/2^{k-k'-1}},$ 
    for any constant $1 > C > 0.$ Indeed,
    $$\lim_{x\downarrow 0^+}\frac{x^{1/2^{k-k'}} - x}{(x^{1/2}-x)^{1/2^{k-k'-1}}} = \lim_{x\downarrow 0^+} \paren{\frac{x^{1/2}}{x^{1/2}-x}}^{1/2^{k-k'-1}} = 1.$$
    So $w_1, \dots, w_{k+1}$ is not $C(\delta^{1/2}-\delta)$-coplanar.
     
    % As the tuple $z_1, \dots, z_{k+1}$ is not $\delta^{1/2}$-coplanar and every sub $(k'+2)$-tuple is not $\delta^{1/2^{k-k'}}$-coplanar
    % As the tuple $z_1, \dots, z_{k+1}$ is not $\delta^{1/2}$-coplanar, the tuple $w_1, \dots, w_{k+1}$ is not $(\delta^{1/2} - \delta)$-coplanar. Otherwise, if $H_w$ is a $(k-1)$-plane where for all $i$ $\dist(w_i, H_w) < \delta^{1/2} - \delta,$ then $\dist(z_i, H_w) \leq \dist(z_i, w_i) + \dist(w_i, H_w) \leq \delta^{1/2},$ a contradiction.
    
    \begin{figure}
        \centering
        \includegraphics[width=0.5\linewidth]{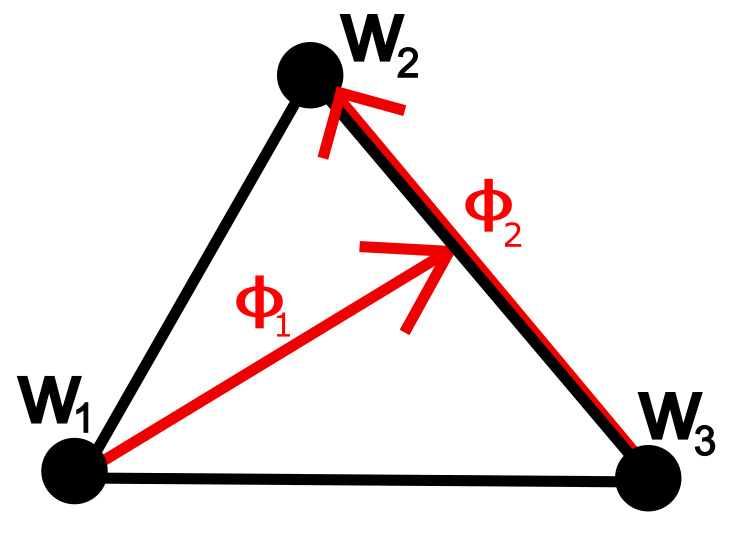}
        \caption{The tuple $w_1, w_2, w_3$ and the orthogonal basis $\{\phi_i\}_{i=1}^3$ generated by them.}
        \label{fig:displacementBasis}
    \end{figure}

    Let $\phi_{i}$ be the displacement vector from $w_i$ to the projection of $w_i$ onto the plane generated by $w_{i + 1}, \dots, w_{k}.$ 
    This yields an orthogonal set of vectors $\{\phi_{i}\}_{i=1}^{k-1}$ that span $H_0$ as shown in Figure~\ref{fig:displacementBasis}, as $\set{w_i, \dots, w_{k+1}}$ is not $C(\delta^{1/2}-\delta)$-coplanar, $\phi_i$ has magnitude at least $C(\delta^{1/2} - \delta)$. 
    
    As $f$ is linear,
    \begin{align*}
        |f(\phi_i)| &= |f(w_i + \phi_i) - f(w_i)| \\
        &\leq |f(w_i + \phi_i)| + |f(w_i)|\\
        &\leq (\delta) + (\delta) \\
        &= 2\delta.
    \end{align*}
    Above, we used that $f(w_i) \leq \delta$ for all $i$ and that $w_i + \phi_i$ belongs to the convex hull of $w_1, \dots, w_{k+1}$ and therefore $f(w_i + \phi_i) \leq \delta.$ 
    
    Let $e_{k+1}=\brac{ 0, 0,\ldots,1}^T \in \R^{k+1}$ let $z_{k+2}$ be represented as $w_1 + \sum_{i = 1}^{k} \frac{c_i}{\|\phi_i\|} \phi_i + c_{k+1}e_{k+1}.$ Then, 
    \begin{align*}
        |f(z_{k+2})| &= \bigg|f\bigg(w_1+\sum_{i=1}^{k}\frac{c_i}{\|\phi_i\|} \phi_i \bigg)\bigg| \\
        &\leq \abs{f(w_1)}+\sum_{i =1}^{k}\frac{|c_i|}{\|\phi_i\|}|f(\phi_i)| \\
        &< (\delta)+\sum_{i=1}^k\frac{|c_i|}{C(\delta^{1/2} - \delta)}(2\delta) \\
        &= \delta+\frac{2\delta^{1/2}}{C(1-\delta^{1/2})} \sum_{i =1}^k|c_i|.
    \end{align*}
    Lastly, as $z_{k+2}$ is $(\log{N} + \delta)$-close to $w_1,$ $c_i \leq \log{N}  + \delta$ for all $i.$ Hence, 
    $$|f(z_{k+2})| \leq \delta+\frac{2\delta^{1/2}}{C(1-\delta^{1/2})}\cdot k(\log{N} + \delta) = O(\delta^{1/3})\,.$$
    
    That is, $z_{k+2}$ must be $O(\delta^{1/3})$ close to $H$. It follows that the set of points $z \subset \R^d$ which form a $\delta$-coplanar $(k+2)$-tuple with $z_1,\ldots,z_{k+1}$ has volume $O(\delta^{1/3}(\log{N})^{k}) = o(\delta^{1/4}).$
\end{proof}

\begin{lemma}\label{lemma:coplanarBound}
    \new{Let $\Omega$ be a Borel set satisfying $\|\Omega\| \geq 1$. Place $\xi := O((\log{N})^d)$ points i.i.d. uniformly in $\Omega.$ Let $T_M$ be the event that there are more than $M$ points of $\Omega$ belonging to $\delta$-coplanar $(k+2)$-tuples. Then, $\mathbb{P}(T_M) =o(\delta^{M/5(k+3)})$.}
\end{lemma}

\begin{proof}
    If $T_M$ occurs, then there exists a subset $I \subset [\xi]=\{1,2,\dots, \xi\}$ where $|I| \leq M + k + 1$, an order on $I$, and a subset $J \subset I$ with $|J| \leq \frac{M}{k+2}$ so that the following occurs: for all $j \in J$, $z_j$ belongs to a $\delta$-coplanar $(k+2)$-tuple with previously placed points. Let $C_J$ denote this event. Our argument, borrowed from ~\cite{Bollob_s_2005}, is to compute a crude upper-bound on the number of possible $I$, orders on $I$, and choices of $J$, and to multiply this with the conditional probability $\mathbb{P}(C_{J} \,| \, I, \text{order on $I$})$. 
    
    First, we compute an upper-bound for the number of choices of $I$. There are precisely $\sum_{i = 0}^{M+k+2} \binom{P}{i}$ subsets $I$ where $|I| \leq M + k+1$. By using the upper-bound $\binom{n}{i} \leq (\frac{en}{i})^i$, we obtain
    \[\sum_{i = 0}^{M + k+1}\binom{\xi}{i} \leq \sum_{i=0}^{M+i}\bigg(\frac{e\xi}{i}\bigg)^i \leq \xi^{M + k+1} \sum_{i =0}^{M+k+1}\bigg(\frac{e\xi}{k\xi^{M+k+1}}\bigg)^i \leq \xi^{M+k+2}\,.\]
    Secondly, the number of choices of orders on $I$ is $|I|! \leq |I|^{|I|} = (M + k+1)^{M + k+1}$ and we have at most $2^{M+k+1}$ choices for $J$. Lastly, we compute a bound on the conditional probability $\mathbb{P}(C_{J} \, | \, I, \text{order on $I$})$.
    
    \new{
    For each $j \in J$, the conditional probability that $z_j$ forms a $\delta$-coplanar $(k+2)$-tuple with previously placed points in $I$ is $$O\paren{\binom{I}{k+2}\frac{\delta^{1/4}}{\|\Omega\|}} = o(\delta^{1/5})$$
    by the uniform distribution of $z_j$, hypothesis $\|\Omega\| \geq 1$, and previous Lemma. So, $\mathbb{P}(C_J \, | \, I, \text{order on $I$}) = o(\delta^{|J|/5}) \leq o(\delta^{M/5(k+2)})$.}

    Putting it all together,
    \begin{align*}
        \mathbb{P}(T_M) &= \overbrace{O((\log{N})^d)^{(M+k+2)}}^{\text{Choices of $I$}} 
        \overbrace{(M+k+2)^{M+k+1}}^{\text{Orders on $I$}}
        \overbrace{2^{M+k+1}}^{\text{Choices of $J$}}
        \mathbb{P}(C_J \, | \, I, \text{order on $I$}) \\
        &= O
        \bigg(\delta^{M/5(k+2)} (\log{N})^{d(M+k+2)}\bigg) \\
        &= o(\delta^{M/5(k+3)}).
    \end{align*}
\end{proof}

\begin{corollary}\label{corollary:coplanarBoundComplete}
    \new{There exist constants $M(k)$ that if $B_{3} = B_{3}(\Lambda)$ is the event that that $\Lambda$ contains more than $M(k)$ $\delta^{1/2}$-coplanar $(k + 2)$-tuples for at least one of $k = 1, \ldots, d-1$ then $\mathbb{P}_p\paren{B_{3,k}}=o(N^{-d})$.}
\end{corollary}

\new{Note any $z_1, \dots, z_{k+1}$ $\delta$-coplanar tuple is $\delta^{\alpha}$-coplanar for $\alpha < 1$ as $\delta^{\alpha} \gg \delta$.}

\begin{proof}
    \new{We define $M(k)$ by
    $$M(k) := \begin{cases}
       % \nicefrac{2(d+1)}{\epsilon 2^{-(d-k)}} & k = 0 \\
        \nicefrac{2^{d-k+1}(d+1)}{\epsilon} & k = 0 \\
        %\nicefrac{5(k+3)(d+1)}{\epsilon 2^{-(d-k)}} & 1 \leq k \leq d.
        \nicefrac{2^{d-k}5(k+3)(d+1)}{\epsilon} & 1 \leq k \leq d - 1\,.

    \end{cases}$$
    Let $B_{3,k}$ be the event there are more than $M(k)$ $\delta^{1/2^{d-k}}$-coplanar $(k+1)$-tuples. The case $k = 0$ is handled by Lemma \ref{lemma:pairBound},
    which asserts $\mathbb{P}_{p}(B_{3,0}) = O(N^{-(d+1)}) = o(N^{-d}).$ For $1 \leq k \leq d-1,$ by the previous Lemma, 
    $$\mathbb{P}(B_{3,k}) = o\paren{\paren{\delta^{1/2^{d-k}}}^{\frac{5(k+3)(d+1)}{5(k+3)} \cdot \frac{2^{(d-k)}}{\epsilon}}} = o(\delta^{-(d+1)/\epsilon}) = o(N^{-d}).$$
    Moreover,
    \begin{align*}
        \mathbb{P}\paren{\bigcup_{k=0}^{d-1} B_{3,k}} &= \sum_{k=0}^{d-1}\mathbb{P}(B_{3,k}) \\
        &= \sum_{k=0}^{d-1} o(N^{-d}) \\
        &= o(N^{-d}).
    \end{align*}}  
\end{proof}

\begin{lemma}\label{lemma:shell}
If $N$ is sufficiently large and $z_{d+2}$ forms a $\delta$-unstable $(d+2)$-tuple with $z_1,\ldots,z_{d+1}$ then $z_{d+2}$ is contained in a spherical shell centered at the circumcenter of $z_1,\ldots,z_{d+1}$ with inner radius $r-\delta^{2/3}$ and outer radius $r+\delta^{2/3},$ where $r$ is the circumradius of $z_1,\ldots,z_{d+1}.$ 
\end{lemma}
\begin{proof}
Recall that $z_1,\ldots,z_{d+2}$ forms a $\delta$-unstable separated tuple if it has diameter at most $2\log{N}$, it contains no $\delta^{1/2}$-coplanar $(d+1)$-tuples, and there is a spherical shell $\mathcal{S}'$ with inner radius $\rho<\log\paren{N}$ and outer radius $\rho + \delta$ containing $z_1, \dots, z_{d+2}.$ Let $S$ be the circumsphere of $z_1,\ldots,z_{d+1}$ and suppose without loss of generality that it is centered at the origin and that $\mathcal{S}'$ is centered at at $\paren{a,0,\ldots,0}.$ We will show that $a$ is close to $0$ and that $\rho$ is close to $r$ using that $z_1,\ldots,z_{d+1}\in S,\mathcal{S}'.$ 

If $w=\paren{w_1,\ldots,w_d}\in S,\mathcal{S}'$ and $W= \sum_{i=2}^d w_i^2$ then
$$w_1^2+W=r^2$$
and
$$\rho^2\leq \paren{w_1-a}^2+W\leq \paren{\rho+\delta}^2$$
so
\begin{equation}
\label{eq:638}
\rho^2\leq a^2-2a w_1+r^2\leq  \paren{\rho+\delta}^2
\end{equation}
and
$$0\leq \frac{r^2+a^2-\rho^2}{2a}-w_1\leq  \frac{2\delta\rho+\delta^2}{2a}\,.$$
That is, $z_1,\ldots,z_{d+1}$ are within distance $\frac{2\delta\rho+\delta^2}{2a}$ of the hyperplane
$$x_1=\frac{r^2+a^2-\rho^2}{2a}\,.$$
Since $z_1,\ldots,z_{d+1}$ is not  $\delta^{1/2}$-coplanar, we must have that
$$\frac{2\delta\rho+\delta^2}{2a}>\delta^{1/2}\implies a<\delta^{1/2}\rho+\nicefrac{1}{2}\delta^{3/2}=O\paren{\log\paren{N}\delta^{1/2}}\,.$$
Rearranging Equation~\ref{eq:638} in a different fashion, we obtain
$$2a w_1-a^2\leq r^2 - \rho^2\leq 2\rho\delta +\delta^2+2a w_1-a^2$$
and $\abs{r^2-\rho^2}=O\paren{\log\paren{N}^2\delta^{1/2}}.$
Now, let $w=\paren{w_1,\ldots,w_d}\in\mathcal{S}'$ be a point which we no longer assume is in $S.$ By a similar computation as before,
\begin{align*}
&\rho^2\leq \paren{w_1-a}^2+W\leq \paren{\rho+\delta}^2 \\
&\implies\\
&r^2+\paren{\rho^2-r^2}\leq w_1^2-2aw_1+W\leq r^2 +\paren{\rho^2-r^2}+2\delta\rho+\delta^2\\
&\implies\\
&\paren{\rho^2-r^2}+2aw_1-w_1^2 \leq d\paren{0,w}^2-r^2\leq \paren{\rho^2-r^2}+2\delta\rho+\delta^2+2aw_1-w_1^2\\
&\implies \\
&\abs{d\paren{0,w}^2-r^2}=o\paren{\delta^{1/3}}
\end{align*}
and the desired conclusion follows.
\end{proof}

\begin{lemma}\label{lemma:tupleBound}
    Let $\Omega$ be a Borel set with $\|\Omega\| \geq 1$. Place $\xi := O((\log{N})^d)$ points i.i.d. uniformly in $\Omega.$ Let $Q_M$ be the event that there are more than $M$ points of $\Omega$ belonging to $\delta$-unstable separated $(d+2)$-tuples. Then, $\mathbb{P}(Q_M) =o(\delta^{M/(d+2)})$. 
\end{lemma}

\begin{proof}
    \new{We follow the same proof as Lemma \ref{lemma:coplanarBound}, but we replace $k+1$ with $d+2$, and modify the computation of the conditional probability which we write below.

    Let $C_J$ be the event that every point $z_j, j \in J$ forms a $\delta$-unstable $(d+2)$-tuple with points indexed by $I$. By the previous Lemma, the volume of the set $\set{x \mid x_1, \dots, x_{d+1}, x \text{ is a $\delta$-unstable separated tuple}}$ is $O(\delta^{2d/3}) = o(\delta).$
    So, the conditional probability of $C_J$ given every previously placed point is $$o\paren{\binom{|I|}{d+2}\delta}/\|\Omega\| = o(\delta)$$ by the uniform distribution of $z_j$, volume computation, and the hypothesis $\|\Omega\| \geq 1$. So, $$\mathbb{P}(C_J \, | \, I , \text{order on $I$}) = o(\delta^{|J|}) = o(\delta^{M/(d+2)})\,.$$ Repeating the computation from the previous proof yields the desired result.}
\end{proof}

\begin{corollary}
    \new{Let $B_4(\Lambda)$ be the event that $\Lambda$ contains more than $M = \frac{(d + 3)(d + 1)}{\epsilon}$ $\delta$-unstable $(d + 2)$ tuples. Then, $\mathbb{P}(B_4) = o(N^{-d})$.}
\end{corollary}

% \begin{proof}
%     By the previous lemma and Lemma \ref{lemma:chernoff},
%     $\mathbb{P}(B_3|B_{1}^c) = o(\delta^{\frac{d+1}{\epsilon}}) = o(N^{-d})$.
% \end{proof}

% So far, we have shown that both $\delta$-close pairs and $\delta$-bad $(d+2)$ tuples are constant in $\Lambda$. As shown in Figure 10 of \cite{Bollob_s_2005}, there may be $\Theta(\frac{\log{N}}{\log{\log{N}}})$ $\delta$-bad points in $\Lambda$. Luckily, a quick application of the Maximum Degree Theorem \cite{bonnet2018maximaldegreepoissondelaunaygraph} shows that this is indeed the worst case scenario.

\begin{corollary}\label{corollary:badpointbound}
    Let $1 > \eta > 0$ be fixed. Let $G=G(\eta)$ be the event that \new{$|Z_\delta| = O(1)$ and $|\mathrm{Star}(Z_\delta \cap \Lambda)| = o(\log{N})$} for every cube $\Lambda$ of length $\eta \log{N}$. Then, $\mathbb{P}(G) = 1 - o(1)$.
\end{corollary}

\begin{proof}
    Let $\Lambda'$ be a cube of length $\log{N}$. If $B_2(\Lambda')$, $B_{3}(\Lambda')$, and $B_4(\Lambda')$ do not occur, then $|Z_\delta \cap \Lambda'|=O\paren{1}.$ Moreover, we may cover $\T^d_N$ with $o(N^d)$ cubes of length $\log{N}$ so that every cube $\Lambda$ of length $\eta \log{N}$ sits inside a cube $\Lambda'$ from the covering. By the union bound, the probability that $B_2(\Lambda)^c \cap B_{3}(\Lambda)^c \cap B_4(\Lambda)^c$ occurs for all $\Lambda$ in the covering is $1 - O(N^d) \cdot o(N^{-d}) = 1 - o(1)$.
    
    Since each cube $\Lambda$ of length $\eta \log{N}$ is contained in at least one cube $\Lambda'$ we may conclude that  $|Z_\delta \cap \Lambda| =O(1)$ with high probability. Lastly, by Corollary \ref{corollary:logarithmBound}, with probability $1 - o(1)$, every vertex of $Z$ has degree $o(\log{N})$. Thus,
    \[|\mathrm{Star}(Z_\delta \cap \Lambda)| = o(\log{N})\]
    with high probability.
\end{proof}

We are now ready to prove Lemma \ref{lemma:clusterBound}.  

\begin{proof}[Proof of Lemma \ref{lemma:clusterBound}]
    \new{Let $1 < \eta < 0.$ Fix a cube $\Lambda$ of length $\eta \log{N}$ and center it at some $z \in C$. Suppose that $G\paren{\eta}$ occurs --- that is that $|Z_\delta \cap \Lambda| = O(1)$ --- and that for all $z \in Z$, $\mathrm{Star}(z)$ has diameter at most $2 \rho \sqrt[d]{\log{N}}.$ By Corollary~\ref{corollary:badpointbound} and Lemma~\ref{lemma:emptyDisc} these events occur with probability $1-o(1).$ Since $C$ is a maximal component of $\mathrm{Star}(Z_\delta)$, the diameter of $C\cap \Lambda$ is at most the diameter of $\mathrm{Star}(Z_\delta\cap \Lambda)$ which is $O(1) \cdot 2\rho \sqrt[d]{\log{N}}$. Therefore, $C \subset \Lambda$ when $N$ is sufficiently large, as $\Lambda$ has length $\eta \log{N} >> 2 \rho \sqrt[d]{\log{N}} \cdot O(1)$. Thus $\abs{C}\leq |\mathrm{Star}(Z_\delta) \cap \Lambda| = o(\log{N})$ with high probability.}
\end{proof}

\section{Topology of Voronoi Tessellations}\label{section:voronoitopology}
We investigate the consequences of the Delaunay stability theorems of \cite{bonnet2018maximaldegreepoissondelaunaygraph} for Voronoi tessellations. As such, the results of this section are stated deterministically under certain hypotheses that will, in later arguments, hold with high probability. 

\subsection{Background}
\begin{definition}
    Let $Z \subset \mathbb{T}_N^d$ be a finite point set. Let $Q \subset Z$ and $z \in Z$. We define
    \begin{enumerate}
        \item the Voronoi cell of $z$,
        \[U(z, Z) = \{x \in \mathbb{T}_N^d : d(x, z) \leq d(x, z') \hspace{0.2cm} \forall z' \in Z\}.\]
        \item the cell complex of Voronoi cells generated by $Q$ in $Z$ as $V(Q,Z)$.
        %\[V(Q,Z) := \{V(q,Z)\}_{q \in Q}\]
        \item the union of Voronoi cells generated by $Q$ in $Z$
        \[U(Q,Z) := \bigcup_{q \in Q} U(q, Z)\] 
        \item the interior site boundary of $Q$
        \[\partial Q := \{q \in Q : \mathrm{Star}(q) \not\subset Q\}\]
        \item the interior of $Q$
        \[\mathrm{Int} \, Q = Q \setminus \partial Q\]
    \end{enumerate}
\end{definition}

In the previous section, we discussed the asymptotics of bad geometries, e.g $\delta$-bad points and clusters. Here, we introduce our notions of good geometry.

\begin{definition}[$\delta$-good and $\delta$-bubble wrapped]
    If $z \in Z$ is not $\delta$-bad, then we call it \textbf{$\delta$-good}. Given a set $Q\subset Z$, we say $Q$ is \textbf{$\delta$-bubble wrapped} if $\mathrm{Star}(Q) \setminus Q$ is $\delta$-good. 
\end{definition}

The utility of bubble wrap, as we will see, is that its topology is stable to reasonably small perturbations. And, just like real-world bubble wrap, while the fragile interior may break, the bubble wrap remains! The goal of this section is to prove the following Theorem.

\begin{theorem}\label{theorem:perturbwrap}
Let $Z \subset \T_N^d$ in general position and let $Q\subset Z$ be $\delta$-bubble wrapped. Let $\zeta$ be a $\delta^5$-perturbation and let $C$ be the union of all $\delta$-bad points of $Z$. If $Z \setminus C$ is a $\log{N}$-sample set, then $\zeta|_{Q\setminus \mathrm{Int}\,C}$ can be extended to a homeomorphism $\bar{\zeta}:\mathbb{T}^d\to\mathbb{T}^d$ so that $\bar{\zeta}\paren{U\paren{Q,Z}}=U\paren{\zeta\paren{Q},\zeta\paren{Z}}\,.$
\end{theorem}

To prove this, we apply a special case of Theorem 4.14 from \cite{boissonnat2013stability}, which we have reformatted using our notation.

\begin{definition}[Combinatorial isomorphism]   Given two (regular) cell complexes $Q_1, Q_2$, we say that a bijection $\zeta$ between the vertices of $Q_1$ and the vertices of $Q_2$ induces \textbf{combinatorial isomorphism} from $Q_1$ to $Q_2$ if $v_1, \dots, v_k$ are vertices of an $i$-dimensional face of $Q_1$ if and only if $\zeta(v_1), \dots, \zeta(v_k)$ are vertices of an $i$-dimensional face in $Q_2.$
\end{definition}

\begin{proposition}{\label{proposition:secureperturbation}}
    Let $Z \subset \mathbb{T}^d_N$ be in general position, $Q \subset Z,$ and $C$ the set of all $\delta$-bad points of $Z$. Further suppose $Z \setminus C$ is a $\log{N}$-sample set. If $Q$ is $\delta$-good and $\zeta : Z \to Z'$ is a $\delta^{5}$ perturbation, then $\zeta$ induces a combinatorial isomorphism between $\mathrm{Star}(Q ; \mathrm{Del}(Z))$ and $ \mathrm{Star}(\zeta(Q) ; \mathrm{Del}(Z'))$. 
\end{proposition}

\begin{proof}
This follows from combining Theorem 4.14 with Lemma 4.10 of \cite{boissonnat2013stability}, using that every simplex of $\mathrm{Del}(Z \setminus Q)$ is not $\delta$-unstable.
\end{proof}

We review our notation here. We call a point $\delta$-unstable if it belongs to either a
$\delta$-coplanar tuples or a separated $\delta$-unstable tuples. 
Points adjacent to $\delta$-unstable points in the Delaunay triangulation are called $\delta$-bad; all other points are $\delta$-good. %From above, we observe that they possess structure preservation under tiny perturbation as shown above. 

In a nutshell, good points are good, bad points are bad, and unstable points are worse. The previous section (Lemma \ref{lemma:clusterBound}) revealed that unstable points are asymptotically few and far between but bad points are non-negligible. As unstable points induce bad points, deleting unstable points reveals a set of entirely $\delta$-good points. %This is the key trick we will utilize in proving Theorem \ref{theorem:perturbwrap}.

\subsection{Stability of Voronoi cells}

\begin{lemma}
    Let $q \in \mathrm{Int}\,Q$. Then, 
    \[U(Q, Z) = U(Q \setminus \{q\}, Z \setminus \{q\}).\]
    Moreover, the boundary does not change. That is, $\partial Q = \partial (Q \setminus \{q\})$.
\end{lemma}

\begin{proof}
    $U(Q \setminus \{q\}, Z) \subset U(Q \setminus \{q\}, Z \setminus \{q\})$. Let $x \in \mathrm{Int} \, V(q, Z)$ and suppose $x \in V(y,Z \setminus \{q\})$. Since $\mathrm{dist}(x,q) < \mathrm{dist}(x,y)$, there exists a point $v$ on the line segment $L$ between $x$ and $y$ that is equidistant to $y$ and $q$. $V(q, Z \setminus \{q\})$ is convex so $L \subset \mathrm{Int} \, V(z, Z \setminus \{q\})$ and $\mathrm{dist}(v,q) = \mathrm{dist}(v,y) < \mathrm{dist}(v, Z \setminus \{y,q\})$. Thus, $v \in V(y, Z) \cap V(q, Z)$ so $y \in \mathrm{Star}(q)$ which implies $y \in Q$ and $x \in U(Q \setminus \{q\}, Z \setminus \{q\})$.

    The second statement directly follows from the observations that for $q'\in Q\setminus \set{q}$ $U\paren{q',Z\setminus q}\subset U\paren{q',Z}\cup U\paren{q,Z}$ and that
    $y \in \mathrm{Int} \, Z' \iff U(y, Z') \subset \mathrm{Int} \, U(Z', Z)$  for any $Z'\subset Z$ and $z\in Z',$ where we are using two different notions of interior. 
\end{proof}

If we iteratively throw away every interior point, the union of Voronoi cells does not change! 

\begin{corollary}\label{cor:interiorDiscard}
    \new{A union of Voronoi cells is invariant under discarding its interior from the point set. That is,}
    \[U(Q,Z) = U(\partial Q, Z \setminus{\mathrm{Int}\,Q}).\]
\end{corollary}

\begin{lemma}\label{lemma:combisoOnVoronoi}
    Let $\zeta : Z \to Z'$ be a $\delta^5$-perturbation where both $Z$ and $Z'$ are in general position, let $C$ be the set of $\delta$-bad points of $Z,$ and assume that $Z \setminus C$ is a $\log{N}$-sample set. Then the restriction $\zeta\mid_{Z\setminus \mathrm{Int} \, C}$ induces a combinatorial isomorphism between the Voronoi diagrams $V\paren{Z \setminus \mathrm{Int} \, C ; Z \setminus \mathrm{Int} \, C}$ and  $V\paren{Z' \setminus \mathrm{Int} \, \zeta\paren{C} ; Z' \setminus\mathrm{Int} \, \zeta\paren{C}}$.
\end{lemma}

\begin{proof}
    By Proposition \ref{proposition:secureperturbation}, every point in $Z\setminus \mathrm{Int} \, C$ is $\delta$-good relative to the point set $Z \setminus \mathrm{Int} \, C$. This induces a combinatorial isomorphisms on the level of Delaunay triangulations,
    \[\mathrm{Del}(Z \setminus \mathrm{Int} \, C; Z \setminus \mathrm{Int}\, C) \cong \mathrm{Del}(\zeta(Z \setminus \mathrm{Int} \, C); \zeta(Z \setminus \mathrm{Int \, C})),\]
    where we used that $\mathrm{Star}(\hat{Z};\hat{Z}) = \mathrm{Del}(\hat{Z};\hat{Z})$ for any point set $\hat{Z}.$ 
    Furthermore, a combinatorial isomorphism of Delaunay triangulations induces a combinatorial isomorphism of the corresponding Voronoi mosaics by their duality relationship. That is,
    \[V(Z \setminus \mathrm{Int}\,  C ; Z \setminus \mathrm{Int} \, C) \cong V(\zeta(Z \setminus \mathrm{Int}\,  C); \zeta(Z \setminus \mathrm{Int} \, C)).\]
\end{proof}

\begin{lemma}\label{lemma:invariantBoundary}
Assume the hypotheses of the previous lemma, and let $Q \subset Z$ be $\delta$-bubble wrapped. Then the boundary of $Q$ is invariant under $\zeta$: 
  \[\partial \zeta(Q) = \zeta\paren{\partial(Q)}\,.\]
\end{lemma}

\begin{proof}
    Let $W = \mathrm{Star}(Q) \setminus Q$. As $Q$ is $\delta$-bubble wrapped, $W$ is $\delta$-good. By Proposition \ref{proposition:secureperturbation}, $\zeta$ induces a combinatorial isomorphism between $\mathrm{Star}(W ; Z)$ and $\mathrm{Star}(\zeta(W); \zeta(Z))$. Thus
    $$\zeta\paren{\partial Q} = \zeta\paren{\mathrm{Star}(W ; Z)\cap Q}=\mathrm{Star}(\zeta(W); \zeta(Z))\cap \zeta\paren{Q}=\partial \zeta\paren{Q}\,.$$
\end{proof}

The last ingredient in our proof of Theorem \ref{theorem:perturbwrap} is to demonstrate the stability of the topology of Voronoi cells of $Z\setminus C$ under $\zeta$. That is, a combinatorial isomorphism should induce a homeomorphism of the torus that maps between the Voronoi tessellations. In particular, topological features like the presence or absence of giant cycles are preserved. We demonstrate this over the next two lemmas.

\begin{lemma}
    Suppose $B_1, B_2$ are homeomorphic to the $d$-ball and $f : \partial B_1 \to \partial B_2$ is a homeomorphism. Then, $f$ may be extended to a homeomorphism $\bar{f}$ from $B_1$ to $B_2$.
\end{lemma}

\begin{proof}
We proceed in cases. 

Case 1: $B_1$ and $B_2$ are unit $d$-balls in $\mathbb{R}^d$ centered at the origin. We may define $\bar{f}:B_1\to B_2$ by
\[\bar{f}(x) :=
\begin{cases}
    \|x\| f(x/\|x\|) & \text{if $x \not = 0$} \\
    0 & \text{else.}
\end{cases}\]
It is a homeomorphism as we may define a continuous inverse, $\bar{f}^{-1}$ by
\[\bar{f}^{-1}(x) : =
\begin{cases}
    \|x\| f^{-1}(x/\|x\|) & \text{if $x \not = 0$} \\
    0 & \text{else.}
\end{cases}\]
As $\bar{f} \circ \bar{f}^{-1} = id$ and $\bar{f}^{-1} \circ \bar{f} = id$, we conclude that $\bar{f}$ is a homeomorphism that extends $f$. 

Case 2: general case. Let $D$ be the unit $d$-ball in $\mathbb{R}^d$ and $g_1 : B_1 \to D, g_2 : B_2 \to D$ be homeomorphisms. Then, $h :=g_2 \circ f \circ g_1^{-1} |_{\partial D} : \partial D \rightarrow \partial D$ is a homeomorphism. Thus, by the previous case, $h$ extends to a homeomorphism $\bar{h} : D \to D$ such that $\bar{h} |_{\partial D} = g_2 \circ f \circ g^{-1} |_{\partial D}.$ Therefore
\[\bar{f} :=g_2^{-1} \circ \bar{h} \circ g_1 : B_1 \to B_2\]
is a homeomorphism that also satisfies the condition
\[\bar{f} |_{\partial B_1} = g_2^{-1} \circ \bar{h} \circ g_1 |_{\partial B_1} = f.\]
\end{proof}

Recall that a $0$-cell is a vertex, a $1$-cell is an edge, a $2$-cell is a face, and so on. 

\begin{definition}[$i$-skeleton]
    Given a cell complex, the \textbf{$i$-skeleton} is the collection of all $i$-cells.
\end{definition}

As each vertex is a $0$-cell, a bijective map on the vertices is trivially a homeomorphism from the $0$-skeleton to the $0$-skeleton. We would hope that a combinatorial isomorphism between two Voronoi diagrams implies that they are homeomorphic in the ambient space. However, to rule out any topological degeneracies, such as a cell wrapping around the torus of length $N$, we enforce an upper bound on the diameter on every Voronoi cell. 

\begin{lemma}
    Let $\zeta$ be a $\delta^5$-perturbation. If $\zeta$ induces a combinatorial isomorphism of $V\paren{Z;Z}$ with $V\paren{f\paren{Z};f\paren{Z}}$ and the diameter of every Voronoi cell is less than $N$  then $\zeta$ can be extended to a homeomorphism $\bar{\zeta}:\mathbb{T}^d\to\mathbb{T}^d$ so that $\bar{\zeta}\paren{V\paren{z;Z}}= V\paren{\zeta\paren{z},\zeta\paren{Z}}$ for all $z\in Z.$ 
\end{lemma}

\begin{proof}
    Let $V^{i}(Z ;Z)$ denote the $i$-skeleton of $V(Z;Z)$. We induct on $i.$ The base case $i=0$ is trivial.
    
    Suppose $\zeta$ extends to a homeomorphism $\zeta_i:V^i(Z;Z) \to V^i(\zeta(Z); \zeta(Z))$ mapping $i$-cells to $i$-cells. Since the diameter of each Voronoi cell is less than $N$, the $(i+1)$-faces are homeomorphic to $(i+1)$-balls. So, $\zeta_i$ extends to a homeomorphism $\bar{\zeta_i}$ mapping an $(i+1)$-cell of $V^{i+1}(Z;Z)$ to an $(i+1)$-cell of $V^{i+1}(\zeta(Z), \zeta(Z))$ by the previous lemma. As two $(i+1)$-cells may only intersect in an $i$-cell by the definition of a cell complex, we may glue over intersecting $(i+1)$-cells to extends further to a homeomorphism $\zeta_{i+1}$ from $V^{i+1}(Z;Z)$ to $V^{i+1}(\zeta(Z) ; \zeta(Z))$.
\end{proof}

%We gave a very general proof above. However, we may give a more explicit proof as Voronoi diagrams have a very explicit structure. 
Now we prove Theorem \ref{theorem:perturbwrap}.

\begin{proof}[Proof of Theorem \ref{theorem:perturbwrap}]
By Lemma \ref{lemma:invariantBoundary}, $\zeta(\mathrm{Int} \, C) = \mathrm{Int} \, \zeta(C)$. We infer from Corollary \ref{cor:interiorDiscard} that
\[U\paren{\zeta\paren{Q}\setminus \zeta\paren{\mathrm{Int}\,C},\zeta\paren{Z}\setminus \zeta\paren{\mathrm{Int}\,C}}=U\paren{\zeta\paren{Q},\zeta\paren{Z}}.\]
Every vertex of $Z \setminus C$ is $\delta$-good so by Lemma \ref{lemma:combisoOnVoronoi} $\zeta$ induces a combinatorial isomorphism between
$V(Q \setminus \mathrm{Int}  \, C , Z \setminus \mathrm{Int} \, C)$ and $V(\zeta(Q \setminus \mathrm{Int} \, C) , \zeta(Z)\setminus \zeta( \mathrm{Int} C))$ which in turn is the same as $V(\zeta(Q) \setminus \mathrm{Int} \, \zeta(C),   \zeta (Z) \setminus \mathrm{Int} \, \zeta ( C))$ by Lemma \ref{lemma:invariantBoundary}. It follows from the previous result that $\zeta|_{Q\setminus \mathrm{Int}\,C}$ extends to a homeomorphism $\bar{\zeta}:\T^d\to \T^d$ so that 
\begin{align*}
    \bar{\zeta}\paren{U\paren{Q,Z}} &=\bar{\zeta}\paren{U\paren{Q\setminus\mathrm{Int}\,C,Z\setminus \mathrm{Int}\,C}} \\
    &=U\paren{\zeta\paren{Q\setminus\mathrm{Int}\,C},\zeta\paren{Z\setminus \mathrm{Int}\,C}} \\
    &=U(\zeta(Q), \zeta(Z))\,.
\end{align*}
\end{proof}

\section{Coupling construction}\label{section:coupling}
In this section, we prove Theorem \ref{thm:main_technical}. The coupling we construct is a generalization of that featured in Lemma 20 of \cite{Bollob_s_2005} to higher dimensions. The arguments of this section are owed directly to Bollobás and Riordan, modified for our setting. Central to their arguments is \textbf{potential instability}, which can be interpreted for now as ``nearly $\delta$-unstable".

The ingredients of this coupling between $\paren{Z_1,R_1}$ and $\paren{Z_2,R_2}$ are laid out in the following scheme:
\begin{enumerate}
    \item A natural coupling where $Z_1=Z_2,$ points are included in $R_1$ independently with probability $p_1,$ and additional points are colored red in $R_2$ with probability $p_2/p_1$.
    \item A coupling for a region of potential instability that ``crosses over'' unstable regions in $Z_1$ with red and fully stable regions in $Z_2.$
    \item Independence between each region of potential instability.
\end{enumerate}
We apply (2) to each region of potential instability and (1) elsewhere, and define our final coupling as an independent product of probability spaces by (3). The nuance of our coupling is due to (2) as it requires moving the points. We begin by defining the natural coupling.

\begin{definition}[Natural coupling]
    Let $0 < p_1 < p_2 < 1$. We couple $(Z, R_1)$ with $(Z, R_2)$ by coloring each white point of $(Z, R_1)$ red in $(Z,R_2)$ independently with probability $p_2/p_1$.
\end{definition}

This coupling satisfies the condition $R_1 \subset R_2$. While most points of $R_1$ will be $\delta$-good, a small portion will be $\delta$-bad and great care must be taken. We must also specify the conditions under which we can make up for these $\delta$-bad regions by increasing $p$ by $\epsilon$. 

Throughout this section, we use the following convention: we write that a random variable $X$ is $o(f(N))$ (resp. $O(f(N))$) with probability $1-o(1)$ to mean that, for any $\eta > 0$ (resp. there exists $\eta>0$ so that) , $X<\eta f(N)$ with high probability as $N\to \infty.$

% To define our final coupling of Voronoi percolations, we write the probability space as a product space where each factor determines the point locations and colors in a specific region. Outside the bad regions, we use the natural coupling; on them, a different coupling is developed in which points are allowed to move a tiny bit.

\begin{definition}[Allocation scheme]
    Let $l$ be the closest number to $\delta^{1/5}$ such that $n := (N/l)^d$ is an integer. We may tile the torus with $n$ cubes $\Lambda_i$ of length $l$. We couple two Poisson point processes, $Z_1$ and $Z_2$, by setting the number of points in $Z_2$ to be equal to the number of points in $Z_1$ in each cube $\Lambda_i$ and placing them independently. We call the assignment of the number of points in each cube an \textbf{allocation scheme} and a placement of the correct number of points in each cube a \textbf{realization} of the allocation scheme.
\end{definition}

For every coupling of Poisson point processes with the same allocation scheme, there exists an \textbf{induced perturbation} $\zeta : Z_1 \to Z_2$ that matches the points in each cube of $Z_1$ to the corresponding points in $Z_2.$ If a cube has more than one point assigned we match those points arbitrarily. 

\begin{lemma}
    With probability $1 - o(1)$, each $\Lambda_i$ has $O(1)$ points.
\end{lemma}
 
\begin{proof}
    Let $\Lambda$ be a cube of length $l$. The probability that $\Lambda$ contains more than $a$ points is bounded above by:
    \begin{align*}
        \mathbb{P}(|\Lambda| \geq a) &\leq \bigg(\frac{e\|\Lambda\|}{a}\bigg)^{a} e^{-\|\Lambda\|} \\
        &\leq \frac{e^al^{ad}}{a^a} \\
        &= \Theta(\frac{e^a}{a^a}\delta^{ad/5}).
    \end{align*}
    Choosing $a = 5(d+1)/\epsilon$ yields $\Theta(N^{-(d+1)}) \leq o(N^{-d})$ in the last expression. As there are $O(N^d)$ cubes, the probability that one of them contains more than $a$ points is $O(N^d) \cdot o(N^{-d}) = o(1)$.
\end{proof}

A choice of allocation scheme determines some instabilities inherited by all of its realizations. For example, if two points are to be placed into a cube, then no matter how they are placed, they will always form an $l\sqrt{d}$-close pair. Despite this, the probability that they form a $\delta$-close pair is still quite small as $\delta$ is exponentially smaller than $l = \Theta(\delta^{1/5})$. On the other hand, if only one point is placed into a cube and every adjacent cube is empty, then that point cannot form a $\delta$-close pair in any realization. We would like to measure these instabilities on a scale greater than $l$, we pick $\delta' = \delta^{1/30}$. A crucial step in our argument is to compare $\delta$-instabilities in realizations to $\delta'$-instabilities in allocation schemes.

% Since $l=\Theta\paren{N^{-\epsilon/5}},$ the realization of an allocation scheme may create instabilities as we will soon see. Recall that $\delta \ll l \ll \delta'$. We will discuss $\delta'$-instabilities of allocation schemes, and later we will compare them to $\delta$-instabilities in realizations.

\begin{definition}[Potential instability and potential badness]
    A cube in an allocation scheme is \textbf{potentially $\delta'$-unstable} (resp. \textbf{potentially $\delta'$-bad}) if there is a realization in which a $\delta'$-unstable (resp. $\delta'$-bad) point is contained in the cube. 

    Consider the graph on the potentially $\delta'$-bad cubes of an allocation scheme where two cubes are connected by an edge if there exists a realization with a $\delta'$-bad cluster that intersects both cubes. A \textbf{potentially $\delta'$-bad cluster} of an allocation scheme is a maximal component of this graph.
\end{definition}

%We will show that these potential instabilities are asymptotically manageable.
We now introduce the event that makes the coupling of Theorem \ref{thm:main_technical} asymptotically feasible.

\begin{definition}[Coupling feasibility]
    For $\eta > 0$, we define $E_{\mathrm{good}, \eta}$ to be the event that every potentially $\delta'$-bad cluster is $\eta \log{N}$ in size and diameter and contains $O(1)$ potentially unstable points. As $\eta$ is arbitrary, by abuse of notation, we write $E_{\mathrm{good}}$ and choose $\eta$ later.
\end{definition}

We will adapt our arguments from the previous section to show that $E_{\mathrm{good}}$ occurs with high probability. 

\begin{figure}[h]
    \centering
    \includegraphics[width=0.5\linewidth]{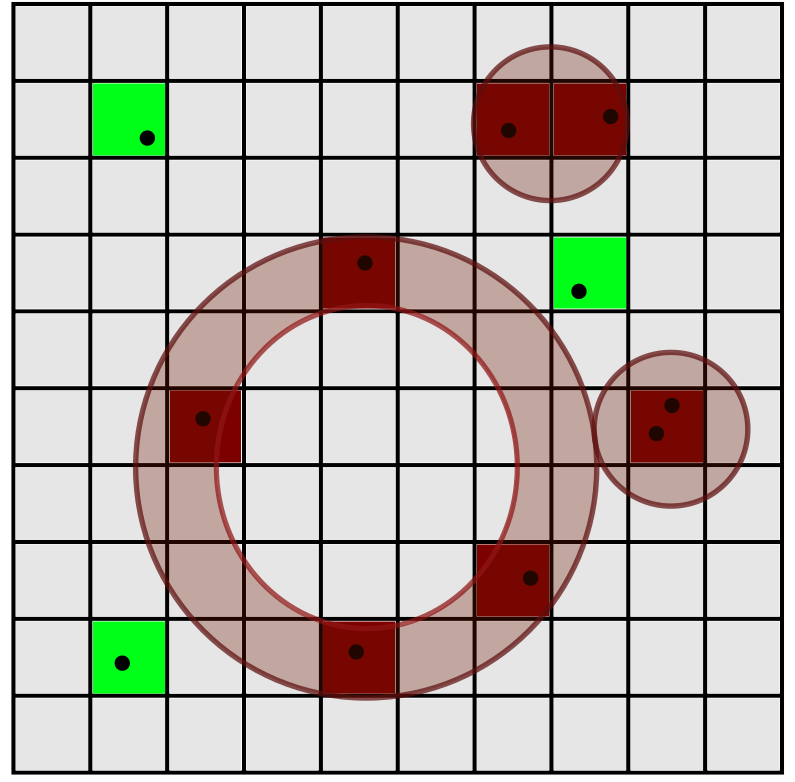}
    \caption{An allocation scheme and a realization with a $\delta'$-unstable $(d+2)$-tuple and two $\delta'$-close pairs. The gray cubes have no points, the red cubes are potentially $\delta'$-unstable, and the green cubes are non-empty and not potentially $\delta'$-unstable. }
    \label{fig:coarsestate}
\end{figure}

A nonempty cube $\Lambda$ is potentially $\delta'$-unstable in two mutually exclusive cases, some subcases are illustrated in Figure~\ref{fig:coarsestate}.
\new{\begin{enumerate}
    % \item ($\delta'$-close pair) A ball of diameter $\delta'$ intersects $\Lambda$ and at least one other nonempty cube.
    % \item ($\delta'$-close pair) $\Lambda$ contains more than one point.
    \item ($\delta'$-coplanar tuple) A $k$-plane of diameter $\log{N}$ and width $\delta'$ intersects $\Lambda$ and at least $k+1$ other nonempty cubes, and these cubes do not belong to the above two cases or this case for any $k' < k$ with $(\delta')^{1/2^{k-k'}}$.
    \item ($\delta'$-unstable tuple) A spherical shell of inner radius $r < \log{N}$ and outer radius $r + \delta'$ intersects $\Lambda$ and at least $(d+1)$ other nonempty cubes and those cubes do not belong to the other three cases.
\end{enumerate}}
Again, we are reiterating the ideas of Section \ref{section:badasymptotics}.

\begin{lemma}\label{lemma:potentiallybadbound}
    When $N$ is sufficiently large, the points in a potentially $\delta'$-bad cube are $2\delta'$-bad in any realization. 
\end{lemma}

\begin{proof}
    First, we argue that the points of a potentially $\delta'$-unstable cube are necessarily $2\delta'$-unstable. We will utilize the facts that $\delta' >> l\sqrt{d} >> \delta$ where $\delta' = \delta^{1/30}$ and $l \sqrt{d} = \Theta(\delta^{1/5})$, as $N\to\infty.$ 

    % Case 1: If a ball of diameter $\delta'$ intersects two nonempty cubes $\Lambda_1$ and $\Lambda_2$ then the distance between a point in $\Lambda_1$ and a point in $\Lambda_2$ is at most $$\delta'+2l\sqrt{d}\approx \delta' +2\delta^{1/5}\sqrt{d}<< 2\delta'.$$ 

    % Case 2: If a cube contains more than one point then it contains a $2\delta'$-close pair.

    \new{Case 1: If a $k$-plane of diameter $\log{N}$ and width $\delta'$ intersects at least $(k+2)$ nonempty cubes, then the points of any realization of those cubes must be at most $(\delta' + l\sqrt{d})$-close to a $k$-plane. So, they form a $2\delta'$-coplanar $(k+2)$-tuple or otherwise contain either a $2\delta'$ $(k'+2)$-coplanar tuple for $k' < k$ or a $2\delta'$-close pair.}

    Case 2: If a spherical shell of inner radius $r$ and outer radius $r + \delta'$ intersects at least $(d+2)$ nonempty cubes then the points of any realization which are contained in those cubes sit in a spherical shell of inner radius $\max\{r - l\sqrt{d}, \delta'\}$ and outer radius $r + \delta' + l\sqrt{d}$. Thus, they contain as a subset either a $2\delta'$-unstable $(d+2)$-tuple if $r>l\sqrt{d}$ or otherwise a $2\delta'$-close pair.

    %Thus, we have shown that the points of a potentially $\delta'$-bad cube must be $2\delta'$-bad in any realization.
\end{proof}

\begin{corollary}\label{cor:couplingfeasibility}
    $$\mathbb{P}(E_{\text{good}}) = 1 - o(1).$$
\end{corollary}

\begin{proof}
    By the previous lemma, the points of every potentially bad cluster are contained in a $2\delta'$-bad cluster in every realization. By Lemma \ref{lemma:clusterBound}, every $2\delta'$-bad cluster has $O(1)$ $2\delta'$-unstable points and is $\eta \log{N}$ in size and diameter with probability $1 - o(1)$.
\end{proof}

Let $C$ be a potentially $\delta'$-bad cluster. For a realization, we say that the points of $C$ are the points placed in cubes contained in $C.$ Such a point is a boundary point if its star contains points in cubes outside of $C.$

\begin{lemma}\label{lemma:boundaryC}
    If $E_{\mathrm{good}}$ occurs, then the boundary points of $C$ are well-defined in the sense that they belong to the same cubes in any realization.
\end{lemma}

\begin{proof}
    Fix an allocation scheme and realization $Z$. Let $C_Z$ be the set of points in $Z$ that belong to $C$. Given another realization of the same allocation scheme, $Z'$, there is an induced perturbation $\zeta : Z \to Z'$ where $\zeta$ is a $l\sqrt{d} << (\delta')^{5}$-perturbation. Let $C$ be the set of all $\delta'$-bad points of $Z$, by $E_{\mathrm{good}}$, $Z \setminus C$ is a $\log{N}$-sample set. Moreover, as $\mathcal{C}_Z$ is a $\delta'$-bubble wrapped set, we may apply Lemma \ref{lemma:invariantBoundary} and obtain: $\zeta(\partial C_Z) = \partial \zeta(C_Z) = \partial C_{Z'}$.
\end{proof}

\begin{definition}
 Let $x_1,\ldots,x_j$ be the interior points of $C$ and $y_1,\ldots,y_k$ be the boundary points of $C$. Define $B\paren{C}$ to be the event that there exist points $y_1',\ldots,y_k'$ in the same $l$-boxes as $y_1,\ldots,y_k$ so that $C'=\set{x_1,...,x_j,y_1',...,y_k'}$ has a $\delta$-bad point.
\end{definition}

By Lemma \ref{lemma:boundaryC}, $B(C)$ only depends on the $\delta$-unstable points of $C$, which are interior points of $C$ and $O(1)$ in number if $E_{\mathrm{good}}$ occurs.

\begin{proposition}\label{prop:Egood} On the event $E_{\mathrm{good}}$ and $N$ sufficiently large, for every potentially $\delta$-bad cluster $C$, $\mathbb{P}(G(C)) \geq 2\mathbb{P}(B(C))$. In particular, there exists $G'(C) \subset G(C)$ and a measure-preserving bijection $f_C:B\paren{C}\cup G'\paren{C}\to B(C) \cup G'(C)$ with $f_C(B(C)) = G'(C)$, $f_C(G'(C)) = B(C)$ whose induced perturbation only moves $\delta$-unstable points.
\end{proposition}

\begin{proof}
Under the event $E_{\mathrm{good}}$ occurs, every potentially bad cluster $C$ has $o\paren{\log\paren{N}}$ points. In particular, if we set
$$\eta = \nicefrac{-2\epsilon}{(10(d+3)\log\epsilon)}$$
then $C$ has fewer than $\eta \log\paren{N}$ points for $N$ sufficiently large. Thus, recalling that $\epsilon=p_2-p_1,$ we may compute an upper bound for the probability of $G(C)$ directly:
    \[\mathbb{P}(G(C)) = \epsilon^{|C|} \geq \epsilon^{\eta \log{N}} = N^{\eta \log{\epsilon}}\,.\]
    Our desired inequality becomes
    \[N^{\eta \log{\epsilon}} \geq 2 \mathbb{P}(B(C))\,.\]
    By our choice of $\eta,$ it will suffice to show that $\mathbb{P}(B(C)) \leq N^{-K\epsilon}$ \new{for some later chosen constant $K.$}

    Let $C_{\mathrm{unstable}}$ denote the union of all potentially $\delta'$-unstable cubes of $C$ and enumerate the cubes of $C_{\mathrm{unstable}}$ as $\Lambda_1, \dots, \Lambda_{T}$ arbitrarily. By $E_{\mathrm{good}}$ there are $O(1)$ potentially $\delta'$-unstable cubes and they have $O(1)$ points. That is, there exist constants $\eta_1, \eta_2$ so that there are at most $\eta_1$ cubes and each has at most $\eta_2$ points for all $N$. Let $1 \leq \lambda_i \leq \eta_2$ be the number of points in the cube $\Lambda_i$.
    
    To apply Lemmas~\ref{lemma:pairBound} and~\ref{lemma:tupleBound}, we rescale the torus by by $1/l$ so that the cubes of $C_{\mathrm{unstable}}$ become unit cubes. Denote by $\mathcal{C}$ the resulting collection of unit cubes. Then, there is a $\delta$-bad point in the original configuration if and only if there is a $\delta_1$-bad point in the rescaled one where $\delta_1 = \delta / l^d$ where we recall that $l \sim \delta^{1/5}$. Let $t = \sum_{i=1}^{T} \lambda_i \leq \eta_2\eta_1$ be the total number of points. If we place $t$ points into $\mathcal{C}$ uniformly (forgetting about the allocation scheme), the probability of a $\delta_1$-close pair is $O(\delta_1^{d\epsilon/2(d+2)})$ by Lemma \ref{lemma:pairBound}. \new{The case for coplanar tuples is similar to the proof of Corollary \ref{corollary:coplanarBoundComplete}. Let $B_{k, \delta_1}$ be the event that $\mathcal{C}$ contains at least one $\delta_1^{1/2^{d-k}}$-coplanar $(k+2)$-tuple. Then,
    \begin{align*}
        \mathbb{P}\paren{\bigcup_{k=1}^d B_{k, \delta_1}} &= \sum_{k=1}^d \mathbb{P}(B_{k,\delta_1}) \\
        %&= \sum_{k=1}^d o(\delta_1^{1/(2^{-d-k}\cdot5(k+2))}) \\
        &= \sum_{k=1}^d o(\delta_1^{2^{k-d}/(5(k+3))}) \\
        &= o(\delta_1^{\kappa}).
    \end{align*}
    for some constant $\kappa(d) > 0$ by Lemma \ref{lemma:coplanarBound}. The probability that a point is in a separated $\delta_1$-unstable $(d+2)$-tuple is $o(\delta_1^{1/(d+3)})$ by Lemma \ref{lemma:tupleBound}}. So, \new{the probability that placing $t$ points i.i.d. uniformly in $\mathcal{C}$ forms at least one $\delta_1$-bad point is $o(\delta^{K})$ for some constant $K(d)>0$.}
    % That is, the probability that placing $t$ points i.i.d uniformly in $C_{\mathrm{unstable}}$ creates a $\delta_1$-bad point is $o(\delta^{K})$
    This implies the desired statement, since there are only finitely many ways to place the the points into cubes and the probability of the correct assignment is bounded away from zero by a constant that does not depend on $N.$ 
\end{proof}

Now, we are ready to begin constructing our coupling. We do so independently for each allocation scheme. If the event $E_{\mathrm{good}}$ fails, we define the coupling arbitrarily. Otherwise, we construct the coupling separately for each potentially $\delta'$-bad cluster. As the assignments of point locations and colors to disjoint potentially bad clusters are independent, we may combine these couplings to define a global coupling for the allocation scheme on the event $E_{\mathrm{good}}.$

More formally, let $Z_1, Z_2$ Poisson point processes coupled to have the same allocation scheme, and condition on a fixed allocation scheme $\mathcal{A}.$ and enumerate the clusters of potentially $\delta'$-bad cubes as $C_j$ for $j = 1,\dots, n$. Let $D_j$ denote the union of the cubes in $C_j$, let $D_0$ be the union of the remaining cubes, let $Z_j^k = Z_j \cap D_k$ and $R_j^k = R_j \cap D_k$, $j = 1,2$. As each $Z_j^k$ are disjoint, we can construct a coupling independently for each $Z_j^k$, $k = 1, \ldots, n$. This partitions our probability space $\Omega$ into a product of probability spaces, $\Omega_0 \times \ldots \times \Omega_n$, where each $\Omega_j$ is the assignment of point locations and colors to the points in $D_j$ for $j = 0, \dots, n$. Lastly, the measure-preserving bijection in the previous proposition, $f_{C_j} : B(C_j) \cup G'(C_j) \to B(C_j) \cup G'(C_j)$, may be equivalently defined on the event space of each $\Omega_j$. We write $f_j$ to be this measure-preserving bijection on $\Omega_j$.

\begin{definition}[Coupling on a cluster]\label{defn:couplingcluster}
Let $Z_1,Z_2$ be as in the previous paragraph, let $0 < p_1 < p_2 < 1,$ and assume that $\mathcal{A} \in E_{\mathrm{good}}.$ Let $(Z_1, R_1^*(p_1))$ and $(Z_2^*, R_2^*(p_2))$ be given by the natural coupling with the allocation scheme $\mathcal{A}$. We construct a coupling of $(Z_1^k, R_1^k)$ with $(Z_2^k, R_2^k)$ over the probability space $(\Omega_j, \mathbb{P}_j)$ by    `
    \[(Z_2^k, R_2^k) = \begin{cases}
        (Z_2^*(\omega), R_2^*(\omega) & \omega \not\in B(C_k) \cup G'(C_k) \\
        (Z_2^*(f_k(\omega)), R_2(f_k(\omega)) & \omega \in B(C_k) \cup G'(C_k). 
    \end{cases}\]
    In addition, define $Q_k$ by
    \[W_k=\begin{cases}
        Z_k^2 & \text{ if $\omega \in B(C_k)$} \\
        R_k^2 & \text{ else.}
    \end{cases}\]
    
    % Let $0 < p_1 < p_2 < 1$ and $\mathcal{A} \in E_{\mathrm{good}}$. Let $Z_1^*, Z_2^*$ be coupled Poisson point processes over $\mathcal{A}$ and let $(Z_1^*, R_1^*)$ and $(Z_2^*, R_2^*)$ be given by the natural coupling. We construct a coupling of $(Z_1^k, R_1^k)$ with $(Z_2^k, R_2^k)$ over the probability space $(\Omega_j, \mathbb{P})$ by
\end{definition}

We discuss the explicit properties of the coupling below.

Let $\zeta$ denote an induced perturbation from $Z_1^k$ to $Z_2^k$. For every $z \in Z_{1}^k$ and $\omega_k \in \Omega_k$
\begin{enumerate}
    \item ($\omega \in B(C_k) \implies f(\omega) \in G'(C_k)$).
    
    If $z$ is $\delta$-bad, then $\zeta(z)$ is red and not $\delta$-bad. Moreover, if $z$ belongs to $\partial C_k$, then $\zeta(z) = z$. Lastly, $W_k = Z_2^k$ and is $\delta$-good.
    
    \item ($\omega \in G'(C_k) \implies f(\omega) \in B(C_k)$).  
    
    If $G'(C)$ occurs, then $z$ is white and $\zeta(z)$ is possibly $\delta$-bad. Similarly, if $z$ is in $\partial C_k$, then $\zeta(z) = z$. Lastly, $W_k = \varnothing$.
    
    \item ($\omega \not\in B(C_k) \cup G'(C_k)$).
    
    We take $\zeta = \mathrm{id}$. $\zeta(z) = z$ and has the same color. If $z$ is white, then $\zeta(z)$ is red with probability $p_2/p_1$. Lastly, $W_k = R_1^k$ and is $\delta$-good.
\end{enumerate}

Finally, we extend this coupling to the entire product space $\Omega_0 \times \Omega_n$. The coupling is a natural coupling over $\Omega_0$ and a cluster coupling over $\Omega_k$ for $k = 1, \dots, n$.

% Fix an allocation scheme for which $E_{\mathrm{good}}$ occurs, enumerate the potentially $\delta'$-bad clusters $C_1,\ldots,C_n.$ Let $D_j$ be the union of the cubes in $C_j$ for $j=1,\ldots,n$ and let $D_0$ be the union of the remaining cubes. Write $Z^j_k=Z_k\cap D_j$ and let $R^j_k$ be the red points in $Z^j_k$ where $k\in\set{1,2}.$ $Z^j_k$ and $R^j_k$ are independent from each other and for different values of $j.$ As such, we may write the Voronoi percolation on the part of our probability space corresponding to the allocation scheme as a product measure on $\Omega_0\times\ldots\times \Omega_n$ where $\Omega_j$ is the assignment of point locations and colors to the points in $D_j$ for $j=0,\ldots,n.$  %each event $X$ may be decomposed over each $\Omega_i$ such that $\mathbb{P}_{p}(X) = \prod_{i=0}^{n} \mathbb{P}_p(X_i).$

%As $\delta'$-potentially bad clusters are disjoint, the events $B(C)$ and $G'(C)$ are disjoint, so we may extend our coupling everywhere. 

%We may denote our probability space of Voronoi percolations on a point set $Z$ as $(\Omega, \mathbb{P}_p)$. We may partition $Z$ into a disjoint union of $\delta'$-potentially bad clusters. 

\begin{definition}[Final coupling]
Let $Z_1, Z_2$ Poisson point processes coupled to have the same allocation scheme $\mathcal{A}$. If $\mathcal{A} \notin E_{\mathrm{good}}$, define the coupling arbitrarily. Otherwise, let $f_0 : \Omega_0 \to \Omega_0$ be the natural coupling and let $f_k : \Omega_k \to \Omega_k$ be the map defined in Definition~\ref{defn:couplingcluster}. Define $f:\Omega_0\times\ldots\times \Omega_n\to\Omega_0\times\ldots\times \Omega_n$ by $f=\prod_{j=0}^k f_j.$ 
\end{definition}

We now check that this final coupling indeed satisfies the properties of Theorem \ref{thm:main_technical}. We denote by $\zeta$ the point perturbation induced by $f.$  Throughout the proof, we also explicitly construct the point set $W.$ 

\textbf{Note:} The statement of the theorem is for $\epsilon_0, \delta_0 \leq N^{-\epsilon_0}$ and they satisfy the hypotheses of the couplings. This changes nothing and avoids notational ambiguity.

\begin{proof}[Proof of Theorem~\ref{thm:main_technical}]
We may assume that $E_\mathrm{good}$ occurs by Corollary \ref{cor:couplingfeasibility}. 

Let $W = \bigcup W_i$, where $W_i$ is defined above in Definition~\ref{defn:couplingcluster}. Recall by Lemma \ref{lemma:boundaryC}, the boundary is well-defined, by Proposition \ref{prop:Egood} the construction of $f$ yields an induced perturbation $\zeta$ does not move boundary points, and by Corollary \ref{cor:interiorDiscard}, $U(Q, Z) = U(Q \setminus \mathrm{Int \, Q}, Z \setminus \mathrm{Int} \, Q)$. Putting these altogether, we have the following analysis for each $C_k$ with $k>0$:
\begin{enumerate}
    \item If $\omega_k \in B(C_k)$, then $\zeta(R_1) \subset R_2^j =: W_i$ is $\delta$-good. Moreover,
    \[U(R_1^k, Z_{1}) \subset U(Z_1^k, Z_1) = U(Z_{1}^k \setminus \mathrm{Int} \, Z_1^k, Z_1 \setminus \mathrm{Int} \, Z_1^k) = U(W, Z_2) = U(R_2^k, Z_2).\]
    The second equality uses $Z_1^k \setminus \mathrm{Int} \, Z_1^k = Z_2^k \setminus \mathrm{Int} \, Z_2^k$ and that no point of $\partial C_k$ is moved for all $k$. 
    \item If $\omega_k \in G'(C_k)$ then $R_1^k = \varnothing =: W_k$ and
    \[U(R_1^k, Z_1) = U(W_k, Z_1) = \varnothing \subset U(R_2^k, Z_2).\] 
    \item If $\omega_k \notin B(C_k) \cup G'(C_k)$ then $Z_1^k = Z_2^k$, $\zeta |_{Z_1^k} = \mathrm{id}$ and $W_k = R_1^k$ is $\delta$-good.
    \[U(R_1^k, Z_1) = U(W_k, Z_1) \subset U(R_2^k, Z_2).\]
\end{enumerate}
Finally $Z_1^0 = Z_2^0$ and $R_1^0 \subset R_2^0$ so $U(R_1^0, Z_1) \subset U(W_0, Z_2)=U(R_1^0, Z_2)$. 
\end{proof}

\section{Homological Percolation at $d = 2i$.}\label{section:evendimension}
\subsection{Homological percolation ingredients}
Consider Voronoi percolation $P=U\paren{Z,R}$ on $\T_N^d.$ The dual percolation is the union of white Voronoi cells, $P^{\bullet}=U\paren{Z,Z\setminus R};$ it is a Voronoi percolation with parameter $1-p.$ Since $Z$ is in general position, the proof of Lemma C.1 in~\cite{bobrowski2020homological} yields that $U\paren{Z\setminus R,Z}$ is homotopy equivalent to $\mathrm{Int}\,U\paren{Z\setminus R,Z}=\T_N^d\setminus P.$  We may then infer a powerful relation between $P$ and $P^{\bullet}$ using Lemmas C.1 and C.2 from \cite{bobrowski2020homological}.
\begin{lemma}[Bobrowski and Skraba]\label{lemma:DualityLemma}
    Let $\phi_i : H_i(P, \mathbb{F}) \to H_i(\mathbb{T}^d_N, \mathbb{F})$ and  $\psi_i : H_i(P^{\bullet}, \mathbb{F}) \to H_i(\mathbb{T}^d_N, \mathbb{F})$ be the maps on homology induced by inclusion. Then 
   $$\mathrm{rank} \,\phi_{i} + \mathrm{rank} \,\psi_{d-i} = \mathrm{rank} \, H_i(\mathbb{T}^d)\,.$$ In particular, $A$ occurs either in the dual or the primal. Similarly, if $S$ occurs in the dual, then $A$ cannot occur in the primal and vice-versa.
\end{lemma}

%While this lemma is stated for permutohedral site percolation, it holds for Voronoi percolation because the Voronoi diagram is the nerve of the Delaunay triangulation with probability one. See the discussion in~\cite{bobrowski2020homological} immediately before the statement.

From this, we obtain the higher dimensional analogue of the result ``for any square on a Voronoi percolation, the probability of a horizontal crossing at $p = 1/2$ is at least $1/2$".

\begin{corollary}\label{cor:AtOneHalf}
    \[\mathbb{P}_{1/2}(A) \geq 1/2\,.\]
\end{corollary}

We also borrow two more technical results from \cite{duncan2025homological} that allow us to compare the probability of $S$ to $A$. We define the coarse state once again for easy reference.

\begin{figure}
    \centering
    \includegraphics[width=0.95\linewidth]{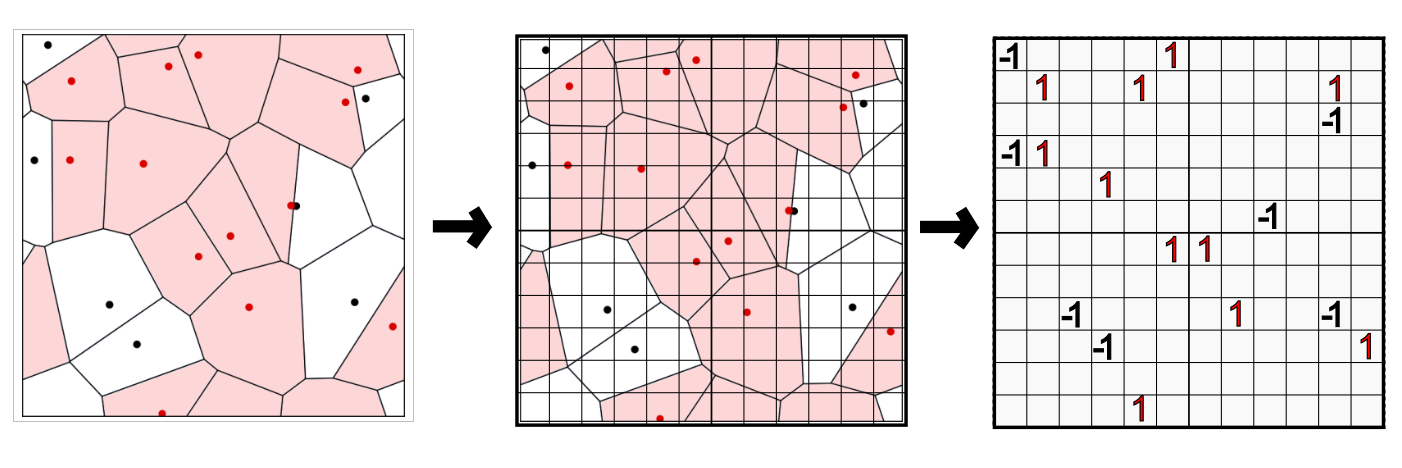}
    \caption{A depiction of a coarse state.}
    \label{fig:CoarsestateVisual}    
\end{figure}

\begin{definition}[Coarse state]\label{def:coarsestate}
    Let $\{\Lambda_{i}\}$ be a disjoint covering of $\mathbb{T}^d_N$ by cubes of length $\delta$. Let $(Z, R)$ be a Voronoi percolation. We define $\mathrm{Coarse}_{\delta}(\cdot) : (Z,R) \mapsto \{-1,0,1\}^{n}$ by, for each cube $\Lambda_i$,
    \[\Lambda_i \mapsto \begin{cases}
        -1 & \text{if $\Lambda_i$ contains a white point} \\
        0 & \text{if $\Lambda_i$ contains no points} \\
        1 & \text{if $\Lambda_i$ contains only red points.}
    \end{cases}\]
\end{definition}

A coarse state is shown in Figure~\ref{fig:CoarsestateVisual}. The first technical result involves an application of representation theory.

\begin{definition}[Irreducible representation]
    A vector space $V$ acted on by a group $G$ is an \textit{irreducible representation} of $G$ if the only subspaces of $V$ are invariant under $G$ or are trivial subspaces.
\end{definition}

\begin{lemma}[Duncan, Kahle, Schweinhart]
    Let $V$ be a finite dimensional vector space and $Y$ be a set. Let $\mathcal{A}$ be a lattice of subspaces of $V$. Suppose $f : \{-1,0,1\}^{|Y|} \to \mathcal{A}$ is an increasing function, i.e. if $A \subset B$ then $f(A) \subset f(B)$. Let $G$ be a finite group which acts on both $Y$ and $V$ whose action is compatible with $f$. That is, for each $g \in G$ and $D \in \mathcal{P}(Y)$ $g(f(D)) = f(gD)$. Let $X$ be a $\{-1,0,1\}^{|Y|}$-valued random variable with a $G$-invariant distribution that is positively associated, meaning that increasing events with respect to $X$ are non-negatively correlated. Then if $V$ is an irreducible representation of $G$, there are positive constants $C_0,C_1$ so that
    \[\mathbb{P}_p(f(X) = V) \geq C_0\mathbb{P}_p(f(X)\not=0)^{C_1},\]
    where $C_0$ only depends on $G$ and $C_1$ only depends on $\mathrm{dim} \, V$.
\end{lemma}

We have modified the above lemma slightly, using $\{-1, 0, 1\}^{|Y|}$ to be the domain of $f$ instead of the power set $\mathcal{P}(Y)$. However, the proof is the same, \textit{mutatis mutandis}. 

%The above lemma, although technical, can be interpreted as a strengthened square-root-trick. Roughly, it states that if a union of increasing and non-trivial events is invariant under some detectable symmetry $G$ has positive probability, then their intersection has positive probability. By ``detectable symmetry $G$", we mean that $V$ is an irreducible representation. If $V$ wasn't irreducible, e.g, it contained a proper and non-trivial $G$-invariant subspace, then the symmetry would be ``undetectable" in that subspace, and we wouldn't be able to exploit the symmetry fully.

We apply the above lemma to events defined in terms of ``stable" giant cycles. Let $\Omega'$ be the collection of all possible coarse states of Voronoi percolations on $\T^d_N.$ Note that  
$\Omega'$ is of the form $\{-1, 0, 1\}^{|Y|}.$ 

\begin{definition}[Homology of a coarse state]
    Let $CS\in \Omega'$ be a coarse state and $P_{CS} =\{P : \mathrm{Coarse}_\delta(P) = CS\}$ the collection of Voronoi percolations with the same coarse state. We define $f : \Omega' \to H_i(\mathbb{T}_N^d, \mathbb{F})$ by $f(CS) = \cap_{P \in P_{CS}} \phi_i(P)$. We call the set, $f(CS)$, the \textbf{stable giant $i$-cycles} of $CS$. 
\end{definition}

In order to apply the lemma, with $V = H_i(\mathbb{T}_N^d; \mathbb{F})$, $Y$ the cubes in the coarse state, $X$ the coarse state of a Voronoi percolation, $f$ as in the previous definition, and $G = W_d$ the point symmetry group of the torus. It remains to show that $V$ is indeed an irreducible representation of $G$. Luckily, this result is provided in \cite{duncan2025homological} as well.

\begin{proposition}
    Let $\mathbb{F}$ be a field, $d > 0$, and $1 \leq i \leq d - 1$. $H_i(\mathbb{T}_N^d)$ is an irreducible representation of $W_d$ if and only if $\mathrm{char}(\mathbb{F}) \not= 2$.
\end{proposition}

Define the events $f(X) = V$ and $f(X) \not= 0$ as $S_\mathrm{discrete}$ and $A_\mathrm{discrete}$ respectively. Thus, we obtain the following uniform bound on $\mathbb{P}_p(S_\mathrm{discrete})$.

\begin{corollary}{\label{corollary:SBound}}
    There are constants $C_0, C_1 > 0$ not depending on $N, i$ such that
    \[\mathbb{P}_p(S_{\mathrm{discrete}}) \geq C_0\mathbb{P}_p(A_{\mathrm{discrete}})^{C_1}.\]
\end{corollary}

\subsection{A sharp transition}
With Theorem~\ref{thm:main_technical} in hand, the proof is a straightforward modification of the arguments of~\cite{duncan2025homological} which substitutes Bollobás and Riordan's version of Theorem \ref{theorem:friedgutkalai} for that of~\cite{Friedgut1996EveryMG}.

Our goal now is to discretize our events into events over the coarse state so that we may apply Theorem \ref{theorem:friedgutkalai}. In the previous subsection, we did this in our ``homology of a coarse state" definition, where the coarse state homology is the common homology of all the percolations with that coarse state. 

\begin{definition}[Stable events and discrete events]
    Let $X$ be an event for Voronoi percolation. Define $X_{\mathrm{discrete}}\subset \Omega'$ to be the event that $X$ occurs for every Voronoi percolation with a given coarse state. Similarly, 
    $X_{\mathrm{stable}}$ occurs for $P=(Z,R)$ if $\mathrm{Coarse}_\delta(P) \in X_\mathrm{discrete}$.
 \end{definition}
The events $A_{\mathrm{discrete}}$ and $S_{\mathrm{discrete}}$ are equivalent to $A_{\mathrm{stable}}$ and $S_{\mathrm{stable}}$, but live on different probability spaces.
%Note that this definition agrees with the one previously given for $A_{\mathrm{discrete}}$ and $S_{\mathrm{discrete}}.$  

Aside from sharing the same name, stable events and stable points are related. For a fixed percolation, $\mathcal{P}$, if an event depends only on the topological features of $\mathcal{P}$ in $\mathbb{T}_N^d$ and depends on a set of $\delta^{1/6}$-good points in $\mathcal{P}$, then $X$ will still hold under any $\delta$-perturbation by Theorem \ref{theorem:perturbwrap}. Subsequently, we formally introduce an event that depends only on topological features.

\begin{definition}[Events invariant under homeomorphisms of the torus]   
    Let $\mathcal{P}$ be a percolation, $\varphi : \mathbb{T}^d \to \mathbb{T}^d$ a homeomorphism, and $\varphi(\mathcal{P})$ another percolation. We say an event $X$ is invariant under homeomorphisms of the torus if $\mathcal{P}\in X$ implies $\varphi\paren{\mathcal{P}}\in X.$
\end{definition}

We express the previous paragraph now compactly as a Corollary of Theorem \ref{theorem:perturbwrap}.

\begin{corollary}
    Let $X$ be an increasing event that is invariant under homeomorphisms of the torus. If there exists an $R' \subset R$ that is $\delta^{1/6}$-stable and $X$ occurs for $(Z,R')$, then $X_{\mathrm{stable}}$ occurs.
\end{corollary}

In this setting, we apply Theorem \ref{thm:main_technical} as follows.

\begin{lemma}
    Let $X$ be an increasing event that is invariant under homeomorphisms of the torus. Then,
    \[\mathbb{P}_p(X_{\mathrm{stable}}) \leq \mathbb{P}_p(X) \leq \mathbb{P}_{p + \epsilon}(X_{\mathrm{stable}}) + o(1)\,.\]
\end{lemma}

\begin{proof}
    The first inequality is trivial as $X_{\mathrm{stable}}$ occurs if $X$ occurs. The second inequality follows from Theorem \ref{thm:main_technical} (by using $\delta_0 =\delta^{1/6}$ in the theorem) and the previous Corollary. That is, if $X$ occurs, as $\zeta(R_1)$ is $\delta^{1/6}$-stable, then $X_{\mathrm{stable}}$ occurs for $(Z_2,R_2)$. 
\end{proof}

The existence of an $i$-dimensional giant cycle, the event $A$, is an increasing event that is invariant under homeomorphisms of the torus. So, the above lemma applies for $A$, which we state as a corollary below.

\begin{corollary}\label{cor:eventStabilization}
    \[\mathbb{P}_p(A_{\mathrm{stable}}) \leq \mathbb{P}_p(A) \leq \mathbb{P}_{p + \epsilon}(A_{\mathrm{stable}}) + o(1)\,.\]
\end{corollary}

We define a proper probability space on the set of coarse states.
\begin{equation}\label{eqn:pbad}
    \begin{cases}
    p_{bad} &=  1 - \exp(-\|B\|(1 - p)) \approx \|B\|(1 - p), \\
    p_{neutral} &= \exp(-\|B\|), \\
    p_{good} &= \mathrm{exp}(-\|B\|(1 - p))(1 - \exp(-\|B\|p)) \approx \|B\|p\,.
    \end{cases}
\end{equation}
% \begin{align}\label{eqn:pbad}
%     p_{bad} &=  1 - \exp(-\|B\|(1 - p)) \approx \|B\|(1 - p), \\
%     p_{neutral} &= \exp(-\|B\|), \\
%     p_{good} &= \mathrm{exp}(-\|B\|(1 - p))(1 - \exp(-\|B\|p)) \approx \|B\|p\,.
% \end{align}
Where $\|B\| = \delta^d$, the size of each cube. Then, as the group of translations on the cubes transitively acts on $\{-1,0,1\}^n$, we are ready to apply Theorem~\ref{theorem:friedgutkalai}.
\[\mathbb{P}_{p}(X_{\text{stable}}) = \mathbb{P}_{p_{bad}, p_{good}}(X_{\text{discrete}}).\]

\begin{lemma}
    If $X$ is an increasing event invariant under homeomorphisms of the torus so that $\mathbb{P}_p(X_{\mathrm{stable}}) > \eta$ for some $\eta > 0$, then $\mathbb{P}_{q}(X_{\mathrm{stable}}) \to 1$ for $q > p$.
\end{lemma}

\begin{proof}
By construction
    \[0 < \mathbb{P}_p(X_{\mathrm{stable}}) = \mathbb{P}_p(X_{\mathrm{discrete}}).\]
    Let $0 < \eta < \mathbb{P}_p(X_{\mathrm{discrete}})$ and $q > p$. Define $q_{bad}, q_{good}$ by substituting $q$ for $p$ in Equation~\ref{eqn:pbad}.  As a reminder, $n$ is the number of cubes in the coarse state and $N$ is the length of the torus.

    We proceed by establishing the inequalities of Theorem \ref{theorem:friedgutkalai}. Throughout this proof, recall that $\delta = N^{-\epsilon}$ and $\|B\| = \nu_d \delta^{d}$
    Firstly, we can demonstrate inequality~\ref{equation:fkinequality1} --- $0 < q_{bad} < p_{bad} < 1/e$ and $0 < p_{good} < q_{good} < 1/e$ --- can quickly be observed as
    \begin{align*}
        &0 <  \|B\|(1-q) < \|B\|(1 - p) < 1/e \\
        &0 <  \|B\|p < \|B\|q < 1/e.
    \end{align*}
    To show inequality \ref{equation:fkinequality2}, observe that $q_{good} - p_{good}, p_{bad} - q_{good} \sim (q - p)\|B\|$ and $p_{\max} = \max\{\|B\|(1-p), \|B\|q\}$. Moreover, using $n = (\lceil \frac{N}{\delta} \rceil)^d$, the right-hand-side of Equation~\ref{equation:fkinequality2} which we denote $\Delta$ has the following upper-bound.
    \new{\begin{align*}
        &\Delta:=c\log{(1/\eta)}p_{\max} \log(1/p_{\max})/ \log{n} \\
        &= (C_1) \cdot p_{\max} \cdot (-\log(p_{\max}))  / \log((N/\delta)^d) \\
        &\leq \frac{C_2}{(d-\epsilon)\log{N}}.
        %& \leq \bigg(c \log{(1/\eta)}\bigg)\bigg(p_{\max}  (-\log{p_{\max})} / \log{((N/\delta)^d)}\bigg) \\
        %&= \bigg(c \log{(1/\eta)}\bigg)\bigg((\nu_d N^{-d \epsilon}\max\{1-p,q\})(-\log{(\nu_dN^{-d\epsilon}  \max\%{1-p,q\})}) / ((d+\epsilon)\log{N}) \bigg) \\
        %&= \bigg(\nu_dc \log{(1/\eta)}\max\{1-p,q\})\bigg)\bigg((N^{-d \epsilon}(d\epsilon\log{N} - \log{\nu_d\max\%{1-p,q\}}) / ((d+\epsilon)\log{N}) \bigg) \\
        %&\leq N^{-d\epsilon} \bigg(C_1 \bigg) \bigg(\frac{d\epsilon}{d+ \epsilon} + \frac{C_2}{\log{N}}\bigg) \\
        %&= C_1\epsilon \|B\| + C_2 \frac{\|B\|}{\log{N}},
    \end{align*}   
    for some constants $C_1, C_2 > 0$. We used that $|p_{\max} (-\log{p_{\max}})| \leq 1$ for $p_{\max} \leq 1.$ As $\min\{q_{good} - p_{good}, p_{bad} - q_{bad}\} \sim (q-p)\|B\|$, $N$ large enough so that $\Delta < q-p$ yields inequality \ref{equation:fkinequality2} of Theorem \ref{theorem:friedgutkalai}}. Therefore, $\mathbb{P}_{q_{good}, q_{bad}}(X_{\mathrm{discrete}}) > 1 - \eta$. It follows that $\mathbb{P}_q(X_{\mathrm{stable}}) = \mathbb{P}_{q_{good}, q_{bad}}(X_{\mathrm{discrete}}) \xrightarrow[]{N \to \infty} 1$.
\end{proof}

\begin{corollary}
    If $X$ is an increasing event invariant under homeomorphisms of the torus and $\mathbb{P}_p(X_{\mathrm{stable}}) > 0$, then $\mathbb{P}_q(X) \to 1$ for $q > p$. 
\end{corollary}

\begin{corollary}
    For $p > 1/2$,
    \[\mathbb{P}_p(A) \to 1.\]
    Moreover, by Corollary \ref{corollary:SBound}, $\mathbb{P}_p(S)$ is uniformly bounded away from $0$.
    \[\mathbb{P}_{p}(S) \to 1.\]
\end{corollary}

When $S$ occurs, the giant cycles block the giant cycles of the dual. That is, by Lemma \ref{lemma:DualityLemma}, we infer the following.
% Add figure here of 1-dimensional giant cycles blocking any potential 1-dimensional giant cycle in the dual

\begin{corollary}
    For $p < 1/2$,
    \[\mathbb{P}_{p}(A) \rightarrow 0\]
\end{corollary}

\section{Acknowledgments}
%Many of the arguments of this work are directly owed to the arguments of Bollobás and Riordan. 
We would like to thank Paul Duncan, Ron Peled, and Anthony Pizzimenti for feedback and helpful comments.

% \section{Sharpness}
% In this section, we want to show
% \[p_c(i,d - 1) < p_c(i, d) < p_c(i+1,d).\]
% First, we show
% \[p_c(i,d - 1) \leq p_c(i,d).\]

% \section{Homological Percolation at $i = 1$ and $i = d- 1$.}

% \begin{theorem}
%     Let $\hat{p_c} = \hat{p_c}(d)$ be the critical threshold for the bond percolation on $\mathbb{Z}^d$. If $i = 1$ then
%     \[\begin{cases}
%         \mathbb{P}_p(A) \rightarrow 0 & p < \hat{p_c} \\
%         \mathbb{P}_p(S) \rightarrow 1 & p > \hat{p_c}.
%     \end{cases}\]
%     as $N \rightarrow \infty$. Furthermore, if $ i = d- 1$ then
%     \[\begin{cases}
%         \mathbb{P}_p(A) \rightarrow 0 & p < 1 - \hat{p_c} \\
%         \mathbb{P}_p(S) \rightarrow 1 & p > 1 - \hat{p_c}.
%     \end{cases}\]
%     as $N \rightarrow \infty$.
% \end{theorem}

\bibliographystyle{plain} % We choose the "plain" reference style
\bibliography{refs} % Entries are in the refs.bib file

\end{document}